\documentclass[12pt]{amsart}
\usepackage{amsfonts,amssymb,eucal}
\usepackage{amsthm}
 \usepackage{amsmath} 
\usepackage{amscd}
 \usepackage{latexsym} 
\input{cyracc.def}

\numberwithin{equation}{section}
\newcommand{\am}{\ensuremath{{\bf{A}}_M}}%
\newcommand{\Hinf}{\ensuremath{\got{H}^{\infty}}}

\newcommand{\Ad}{\ensuremath{{\mbox{\rm{Ad}}}}}
\newcommand{\ad}{{\mbox{\rm ad}}}
\newcommand{\e}{{\mbox{\rm e}}}

\newcommand{\mb}[1]{{\mbox{\boldmath{$#1$}}}}
\newcommand{\mc}[1]{{\mathcal{#1}}}
\newcommand{\got}[1]{{\mathfrak{#1}}}
\newcommand{\db}[1]{{\mathbb{#1}}}

\newcommand{\pa}{\partial}

\newcommand{\R}{\ensuremath{\mathbb{R}}}
\newcommand{\C}{\ensuremath{\mathbb{C}}}
\newcommand{\F}{\ensuremath{\mathbb{F}}}
\newcommand{\N}{\ensuremath{\mathbb{N}}}
  
\newcommand{\oo}{\mathbb{I}}
\newcommand{\un}{{\oo_n}}
\newcommand{\dn}{{\oo_{2n}}}
\newcommand{\zn}{{\mathbb{O}_n}}

\newcommand{\Hi}{\ensuremath{\got{H}}}

\newcommand{\g}{\ensuremath{\got{g}}}
\newcommand{\gc}{\ensuremath{\got{g}_{\C}}}
\newcommand{\Ugc}{\ensuremath{\mathcal{U}({\gc}})}

\newtheorem{Remark}{Remark}

\newcommand{\DM}{\ensuremath{{\got{D}}_M }}
\newcommand{\AM}{\ensuremath{{\got{A}}_M }}
\newcommand{\D}{\ensuremath{{\got{D}}}}
\newcommand{\AAA}{\ensuremath{{\db{A}}_M}}

\newcommand{\fl}{\ensuremath{{\mathcal{F}}_{\Hi}}}

\newtheorem{Proposition}{Proposition}
\newtheorem{lemma}{Lemma}
\newtheorem{corollary}{Corollary}

\theoremstyle{definition}

 \def\i{\mathrm{i}}
\def\ee{\mathrm{e}}
\newcommand{\tr}{\ensuremath{{\mbox{\rm{Tr}}}}}%
\newcommand{\dia}{\ensuremath{{\mbox{\rm{diag}}}}}%
\newcommand{\Ham}{\ensuremath{{\mbox{\rm{Ham}}}}}%
\newcommand{\dd}{\operatorname{d}}
\begin{document}

\title{A convenient coordinatization  of    Siegel-Jacobi domains}

\author{Stefan  Berceanu}
\address[Stefan  Berceanu]{Horia Hulubei National
 Institute for Physics and Nuclear Engineering\\
        Department of Theoretical Physics\\
 P.O.B. MG-6, 077125 Magurele, Romania}
\email{ Berceanu@theory.nipne.ro}


\begin{abstract} 
We determine the homogeneous K\"ahler diffeomorphism $FC$ which
expresses the K\"ahler two-form on the Siegel-Jacobi ball  $\mc{D}^J_n=\C^n\times \mc{D}_n$ 
as the sum of the K\"ahler two-form on $\C^n$ and the one on the
Siegel ball $\mc{D}_n$. The classical motion and quantum evolution  on
$\mc{D}^J_n$ determined by a hermitian  linear Hamiltonian in the generators of
the Jacobi group $G^J_n=H_n\rtimes\text{Sp}(n,\R)_{\C}$ are  described by a
matrix Riccati equation on $\mc{D}_n$ and a linear first order
differential equation in $z\in\C^n$, with coefficients depending also
on $W\in\mc{D}_n$. $H_n$ denotes the $(2n+1)$-dimensional Heisenberg
group. The system of  linear  differential equations
attached to the matrix Riccati equation is a linear Hamiltonian
system on $\mc{D}_n$. When the  transform $FC:(\eta,W)\rightarrow (z,W)$
is applied,  the first order differential  equation in  the variable
$\eta=(\un-W\bar{W})^{-1}(z+W\bar{z})\in\C^n$ becomes decoupled from the
motion on the Siegel ball. Similar considerations are presented for
the Siegel-Jacobi upper   half plane $\mc{X}^J_n=\C^n\times\mc{X}_n$,
where $\mc{X}_n$ denotes the Siegel upper half plane. 
\end{abstract}
\maketitle

\noindent

\setcounter{tocdepth}{3}

\tableofcontents
\newpage

\section{Introduction}

The Jacobi groups \cite{ez} are semidirect products of appropriate semisimple real
algebraic groups of hermitian type with Heisenberg groups \cite{TA99, LEE03}. The Jacobi groups are unimodular,
nonreductive,  algebraic groups of Harish-Chandra type \cite{satake}. The Siegel-Jacobi
domains are nonsymmetric domains associated to the Jacobi groups
by the generalized Harish-Chandra embedding \cite{satake,LEE03}, 
\cite{Y02} - \cite{Y08}.

The holomorphic irreducible unitary representations of the Jacobi groups
based on Siegel-Jacobi domains have been constructed by Berndt, 
B\"{o}cherer, Schmidt, and Takase \cite{bb,bs}, \cite{TA90} - \cite{TA99}.
Some coherent state systems based on Siegel-Jacobi domains have been
investigated in the framework of quantum mechanics, geometric quantization,
dequantization, quantum optics, nuclear structure, and signal processing 
\cite{KRSAR82,Q90,SH03,jac1,sbj,sbcg,gem}.  The Jacobi group was investigated by
physicists under other names as 
 {\it Hagen} \cite{hagen},  {\it Schr\"odinger} \cite{ni},    or {\it Weyl-symplectic} group
 \cite{kbw1}. The Jacobi
group is  responsible  for the   squeezed states  \cite{ken,stol,lu,yu,ho} in quantum optics
\cite{dod,mandel,ali,siv,dr}. 

The  Jacobi group has been   studied  in the papers \cite{jac1,sbj} in 
connection with the group-theoretic approach to coherent states 
\cite{perG}.  We have attached  to the Jacobi group
$G^J_n=H_n\rtimes\text{Sp}(n,\R)_{\C}$ coherent states  based on
Siegel-Jacobi ball  $\mc{D}^J_n$ \cite{sbj}, which, as set, consists
of the points of 
$\C^n\times\mc{D}_n$. $H_n$ denotes the 
$(2n+1)$-dimensional Heisenberg group  and $\mc{D}_n$ denotes the
Siegel ball.    The case
$G^J_1$ was studied  in \cite{jac1}. We have determined the
K\"ahler two-form $\omega_n$ on $\mc{D}^J_n$  from  the K\"ahler
potential  \cite{sbj} -- 
the logarithm   of the scalar product of two coherent  states
\cite{sb6} --  and,  via the partial Cayley
transform, we have determined the K\"ahler two-form $\omega'_n$  on the
Siegel-Jacobi upper-half plane $\mc{X}^J_n=\C^n\times\mc{X}_n$, where
$\mc{X}_n$ is the Siegel upper half plane \cite{mlad}. The K\"ahler two-form
$\omega'_1$ was
 investigated by Berndt \cite{bern}  and K\"ahler
\cite{cal3,cal},  while  $\omega_n$
and  $\omega'_n$ have been investigated also by Yang \cite{Y07} -
\cite{Y10}. $\omega_n$ is written compactly as 
the sum of two terms, one describing  the K\"ahler  two-form on 
$\mc{D}_n$, $\omega_{\mc{D}_n}$,  the other one is 
$\tr(A^t(\un-\bar{W}W)^{-1}\!\wedge\bar{A})$,  where $ A=\dd z +\dd W\bar{\eta}$, and 
$\eta=(\un-W\bar{W})^{-1}(z+W \bar{z})$,  $z\in\C^n, W\in \mc{D}_n$  \cite{mlad,sbj}. Let us
denote by $FC$ the change of variables $FC: ~
\C^n\times\mc{D}_n$$\ni 
(\eta,W)\rightarrow$$ (z,W)\in\mc{D}^J_n$, $z= \eta-W\bar{\eta}$. 
We put
this change of variables in connection with the celebrated fundamental
conjecture \cite{GV,DN} on the Siegel-Jacobi ball $\mc{D}^J_n$, 
as we did in \cite{FC1,FC} for the Siegel-Jacobi disk $\mc{D}^J_1$. We
also make similar considerations
 for the Siegel-Jacobi upper half plane $\mc{X}^J_n$.

In \cite{sbcag,sbl} we have considered the problem of {\it dequantization} of a dynamical system with Lie group of
symmetry $G$   on a Hilbert space
$\Hi$ in 
Berezin's approach \cite{berezin2,berezin1} in the simple case of linear Hamiltonians.
Linear Hamiltonians in generators of the Jacobi group appear in many
physical problems of quantum mechanics, as in the case of the  quantum oscillator
acted on by a variable external force \cite{fey,sw,hs}.

What we find out is that the  classical motion and quantum
evolution  on $\mc{D}^J_n$ determined by  linear Hamiltonians in the generators of the
Jacobi group $G^J_n$ are  
described by a matrix Riccati equation on $\mc{D}_n$ and a first order
coupled  linear differential 
equation for  $z\in \C^n$. The nice thing is that via the $FC$
transform, the differential equation for  $\dot{\eta}$ does not depend on $W\in\mc{D}_n$. The variables $(\eta, W)$ appear to offer a 
convenient parametrization of the Siegel-Jacobi ball.   Similar considerations are
presented for the  equations of motion on $\mc{X}^J_n$. 

In the present paper we  extend to  $G^J_n$ our results established in
\cite{FC} for $G^J_1$.

The paper is laid out as follows. 
The notation for the Jacobi algebra $\got{g}^J_n$   is fixed in  \S
\ref{JMAREALGEBRA}. Starting with some notation on coherent states
\cite{perG,sb6},  \S  \ref{mare} deals with  coherent states
based on $\mc{D}^J_n$ \cite{sbj}. The holomorphic representation
\cite{last,sb6}  of 
the generators of the Jacobi group  as
first order differential operators with polynomial
coefficients defined  on $\mc{D}^J_n$ is given  in \S\ref{DIFFAC}. It is
verified that the differential realization of   the generators of
$\got{sp}(n,\R)_{\C}$ has the right   hermiticity properties with
respect to  the scalar product of functions on $\mc{D}_n$ -  Lemma
\ref{lema2}, and similarly for the generators of $\got{g}^J_n$ - 
Lemma \ref{LEMMA4}. In \S \ref{RELGRJ}  we recall  the expression of
$\omega'_n$ and in Proposition \ref{PARTC}  we show that the partial
Cayley transform -the transform which connects $\mc{D}^J_n$ and
$\mc{X}^J_n$ - is a K\"ahler
homogenous diffeomorphism. In Proposition \ref{LKJ} of \S\ref{FCNN}  we
determine the  $FC$-transform for
the Siegel-Jacobi domains 
$\mc{D}^J_n$ and $\mc{X}^J_n$. The proof is inspired by the paper
\cite{cal3}  of
K\"ahler. Corollary \ref{UNICUC}  expresses  the reproducing kernel and the
scalar product in the
variables $(\eta, W)\in\C^n\times\mc{D}_n$. \S \ref{CLSQ1} is devoted to classical
motion and
quantum evolution on the Siegel-Jacobi domains determined by hermitian
linear
Hamiltonians in the generators of the Jacobi group $G^J_n$.  The
equations of motion are written down explicitly in
Proposition \ref{POYT} in \S \ref{lasst} and their integration
is discussed in \S \ref{solution}. Use is made of the methods of
\cite{levin}  to
integrate  the   matrix
Riccati differential  equation on manifolds  by linearization,
previously 
applied  in \cite{sbl} in the case of hermitian symmetric spaces.   The last two paragraphs in
\S \ref{FAZE} refer to the
Berry phase \cite{swA} on $\mc{D}^J_n$ and the energy function associated to the
Hamiltonian linear in the generators of the Jacobi group $G^J_n$
expressed in the variables $(\eta, W)\in\C^n\times\mc{D}_n$. In these
variables the energy function is written down as the sum of a real  function in
$\eta$ and  one in $W$. 

\textbf{Notation.} We denote by $\mathbb{R}$, $\mathbb{C}$, $\mathbb{Z}$,
and $\mathbb{N}$ the field of real numbers, the field of complex numbers,
the ring of integers, and the set of non-negative integers, respectively. 
$M_{mn}(\mathbb{F})\approxeq\mathbb{F}^{mn}$ denotes the set of all $m\times
n $ matrices with entries in the field $\mathbb{F}$. $M_{n1}(\mathbb{F})$ is
identified with $\mathbb{F}^n$. Set $M(n,\mathbb{F})=M_{nn}(\mathbb{F})$.
For any $A\in M(n,\mathbb{F})$, $A^{t}$ denotes the transpose matrix of
$A$, $A^s=(A+A^t)/2$ and $A^a=(A-A^t)/2$. For $A\in M_{n}(\mathbb{C})$, $\bar{A}$ denotes the conjugate matrix
of $A$ and $A^{*}=\bar{A}^{t}$. For $A\in M_n(\mathbb{C})$, the
inequality $A>0$ means that $A$ is positive definite. The identity matrix of
degree $n$ is denoted by $\un$ and  $\zn$ denotes the $M_n(\F)$-matrix with all entries $0$. We denote by
$\text{diag}(\alpha_1,\dots,\alpha_n$)  the matrix which has the
elements $\alpha_1,\dots,\alpha_n$ on the diagonal and all the other
elements 0.
If $A$ is a linear operator, we denote by
$A^{\dagger}$ its adjoint. We consider a complex separable Hilbert space $\got{H}$ endowed with a   scalar product
which  is antilinear in the first argument,
$<\lambda x,y>=\bar{\lambda}<x,y>$, $x,y\in\got{H}$,
$\lambda\in\C\setminus 0$.
 A complex analytic manifold is
{\it K\"ahlerian}  if it is endowed with a Hermitian metric whose imaginary
part $\omega$  has $\dd\omega = 0$ \cite{helg}. A coset space $M=G/H$  is {\it homogenous
K\"ahlerian} if it caries a K\"ahlerian structure invariant under the
group $G$ \cite{bo}. By a  {\it  K\"ahler homogeneous diffeomorphism}
we mean a 
diffeomorphism $\phi:M\rightarrow N$  of homogeneous
K\"ahler manifolds  such that $\phi^*\omega_N=\omega_M$. 
$\Ham(M)$ denotes the Hamiltonian vector fields on the
manifold $M$. We use Einstein convention  that repeated indices are
implicitly summed over.
If $W=(w)_{ij}$ is a symmetric matrix, we  introduce the symbols $\nabla_{ij}=\nabla_{ji}=\chi_{ij}\frac{\pa}{\pa w_{ij}}$,
where $\chi_{ij}=\frac{1+\delta_{ij}}{2}$.
In the expression  $\chi_{ij}\frac{\pa}{\pa w_{ij}}$ of $\nabla_{ij}$ no summation is assumed.


\section{The Jacobi algebra $\got{g}^J_n$ }\label{JMAREALGEBRA}

Let $\got{h}_n$ denotes the $(2n+1)$-dimensional  
Heisenberg  algebra,  isomorphic to the  algebra 
 \begin{equation}\label{nr00}\got{h}_n =
<\i s 1+\sum_{i=1}^n (x_i a_i^{\dagger}-\bar{x}_ia_i)>_{s\in\R ,x_i\in\C}
,\end{equation} 
 where ${ a}_i^{\dagger}$ (${ a}_i$) are  the boson creation
(respectively, annihilation)
operators,  which verify the canonical commutation relations
\begin{equation}\label{baza1M}
[a_i,a^{\dagger}_j]=\delta_{ij}; ~ [a_i,a_j] = [a_i^{\dagger},a_j^{\dagger}]= 0 . 
\end{equation}
The displacement operator
\begin{equation}\label{deplasareM}
D(\alpha ):=\exp (\alpha a^{\dagger}-\bar{\alpha}a)=\exp(-\frac{1}{2}|\alpha
|^2)  \exp (\alpha a^{\dagger})\exp(-\bar{\alpha}a),
\end{equation}
verifies the composition rule:
\begin{equation}\label{thetahM}
D(\alpha_2)D(\alpha_1)=\ee^{\i\theta_h(\alpha_2,\alpha_1)}
D(\alpha_2+\alpha_1) , 
~\theta_h(\alpha_2,\alpha_1):=\Im (\alpha_2\bar{\alpha_1}) .
\end{equation}
  Here we have used the notation
 $\alpha \beta=  \alpha_i\beta_i$, where $\alpha = (\alpha_i )_{i=1,\dots,n}\in\C^n$.

The composition law of the Heisenberg  group $H_n$ is:
\begin{equation}\label{clh}
(\alpha_2,t_2)\circ (\alpha_1,t_1)=(\alpha_2+\alpha_1, t_2+t_1+
\Im (\alpha_2\bar{\alpha_1})). 
\end{equation}
If  we identify $\R^{2n}$ with $\C^n$, $(p,q)\mapsto \alpha$:
\begin{equation}
\label{change}
\alpha = p +\i q, ~p,q\in\R^n,
\end{equation} 
then $$ \Im (\alpha_2\bar{\alpha_1})=
(p^t_1, q^t_1)J\left(\begin{array}{c}p_2 \\ q_2\end{array}\right),
\quad \text{where} ~
J=\left(\begin{array}{cc} \zn & \un\\ -\un & \zn \end{array}\right) .$$

The Jacobi algebra is the  the semi-direct sum
$\got{g}^J_n:= \got{h}_n\rtimes \got{sp}(n,\R )_{\C}$,
where $\got{h}_n$ is an  ideal in $\got{g}^J_n$,
i.e. $[\got{h}_n,\got{g}^J_n]=\got{h}_n$, 
determined by the commutation relations:
\begin{subequations}\label{baza3M}
\begin{eqnarray}
\label{baza31}[a^{\dagger}_k,K^+_{ij}] & = & [a_k,K^-_{ij}]=0,  \\
~[a_i,K^+_{kj}]  & = &  
\frac{1}{2}\delta_{ik}a^{\dagger}_j+\frac{1}{2}\delta_{ij}a^{\dagger}_k ,~
 [K^-_{kj},a^{\dagger}_i]  = 
\frac{1}{2}\delta_{ik}a_j+\frac{1}{2}\delta_{ij}a_k , \\
~  [K^0_{ij},a^{\dagger}_k] & = & \frac{1}{2}\delta_{jk}a^{\dagger}_i,~
[a_k,K^0_{ij}]= \frac{1}{2}\delta_{ik}a_{j} .
\end{eqnarray}
\end{subequations}

The  generators $K^{0,+,-}$ of $\got{sp}(n,\R )_{\C}$ verify the  commutation
relations 
\begin{subequations}\label{baza2M}
\begin{eqnarray}
 [K_{ij}^-,K_{kl}^-] & = & [K_{ij}^+,K_{kl}^+]=0 , ~2[K^-_{ij},K^0_{kl}]  =  K_{il}^-\delta_{kj}+K^-_{jl}\delta_{ki}\label{baza23}, \\
 2[K_{ij}^-,K_{kl}^+] & = & K^0_{kj}\delta_{li}+
K^0_{lj}\delta_{ki}+K^0_{ki}\delta_{lj}+K^0_{li}\delta_{kj}
\label{baza22}, \\
2[K^+_{ij},K^0_{kl}] & = & -K^+_{ik}\delta_{jl}-K^+_{jk}\delta_{li}
 \label{baza24}, ~
 2[K^0_{ji},K^0_{kl}] =  K^0_{jl}\delta_{ki}-K^0_{ki}\delta_{lj} .  
\end{eqnarray}
\end{subequations}

Now we briefly fix the notation  concerning the symplectic group.
 For $A\in \text{GL}(2n,\F)$, we have (here  $\F$ is any of the
 fields $\R, \C$):
$  A\in \text{Sp}(n,\F )~\leftrightarrow ~ A^tJA=J $. 
Consequently, a matrix $X\in\got{gl}(2n,\F)$ is in $\got{sp}(n,\F)$ iff $X^tJ+JX=
0$. The matrices from  $\got{sp}(n,\R)$ are also called
{\it infinitesimally symplectic} or {\it Hamiltonians} \cite{mey}.

We
recall  also that $g\in \text{U}(n,n)$ iff $gKg^*=K$, where
$K=\left(\begin{array}{cc}
\un & \zn \\ \zn & -\un\end{array}\right)$.

We summarize some properties of symplectic  and Hamiltonian
matrices (cf. \cite{sieg,bar70,fol,mey};
for  the  characterization of  the eigenvalues, see 
\cite{mey,am,lm,ya}):
\begin{Remark}\label{rem77} 

a) $X$ is a Hamiltonian matrix iff one of the following equivalent
conditions are fulfilled:\\
\mbox{~~~~~}1) $X^tJ+JX=0;$\\
\mbox{~~~~~}2) $X=JR$, where  $R\in M(2n,\R)$ is a symmetric matrix;\\
\mbox{~~~~~}3) $JX$ is symmetric;\\
\mbox{~~~~~}4) $X\in M(2n,\R)$ has the form 
\begin{equation}\label{XREAL}
X=\left( \begin{matrix} a & b \\c & -a^t\end{matrix}\right)\in\got{sp}(n,\R),
\quad
b=b^t,\quad c=c^t,  ~~a,b,c\in M(n,\R );
\end{equation} 

b)  If $ M=\left(\begin{array}{cc}a & b\\c& d\end{array}\right) \in
 {\emph{\text{Sp}}}(2n,\R)$, then the matrices $a,b,c,d\in M(n,\R)$ have the properties
\begin{subequations}\label{lica}
\begin{align}
a^tc & = c^ta, ~ b^t d = d^t b, ~ a^td - c^t b =\un ; \label{lica1}\\
ab^t & = ba^t, ~ cd^t=dc^t, ~ ad^t-bc^t =\un . \label{lica2}
\end{align}
\end{subequations}

c) Under the identification \eqref{change} of $\R^{2n}$ with $\C^n$, we
have the correspondence  
\begin{equation}\label{pana}
M=\left(\begin{array}{cc}a & b\\c & d\end{array}\right) \in  M(2n,\R
)\leftrightarrow M_{\C}   = \mc{C}^{-1}M\mc{C}  = \left(\begin{array}{cc} p & q \\ \bar{q}
&\bar{p} \end{array}\right),~ p, q\in M(n,\C ),  
\end{equation}
where \begin{equation}\label{CCC}
\mc{C}=\left(\begin{array}{cc} \i \un & \i \un
\\
-\un  & \un\end{array}\right), ~ \mc{C}^{-1}= \frac{1}{2}
\left(\begin{array}{cc} - \i \un  & -\un \\ -\i \un &
    \un\end{array}\right); \end{equation}
\begin{subequations} \label{CUCURUCU}
\begin{align}
2 a = & p+q +\bar{p}+\bar{q} ,~
2 b = \i(\bar{p}-\bar{q}-p+q) ,\\
2 c = & \i(p+q-\bar{p}-\bar{q}), ~
 2 d =    p-q +\bar{p}-\bar{q}; 
\end{align}
\end{subequations}
  \begin{equation} \label{CUCURUCU1}
2p =  a+d+\i(b-c), \quad 2q= a-d -\i (b+c). 
\end{equation}
In particular, to the Hamiltonian matrix     \eqref{XREAL}  $X$ we associate
 $ X_{\C}=\mc{C}^{-1}X\mc{C}\in \got{sp}(n,\R)_{\C}$$=\got{sp}(n,\C)\cap \got{u}
 (n,n),$
\begin{equation}\label{xXC}
 X_{\C}=\left(\begin{matrix} p & q \\ \bar{q} &
\bar{p}\end{matrix}\right),~ p^*=-p, \quad q^t= q , 
\end{equation}
where
 \begin{equation}\label{MNB}
2p= a-a^t +\i (b-c); ~2q=
a+a^t-\i (b+c) .
\end{equation}

The  relations inverse to \eqref{MNB}
are 
\begin{equation}\label{MNB1}
2 a = p +\bar{p}+q+\bar{q}; 2 b= \i (\bar{p}-p+q-\bar{q}) ; 2 c= \i
(p+q-\bar{p}-\bar{q}) .
\end{equation}

Also, we have 
\begin{equation}\label{ciudat}
X\in\got{sp}(n,\R)_{\C}\quad{\emph{iff}} \quad X=\i K H,~H^{*}=H,~H\in M(2n,\C).
\end{equation}

d) To every $g\in {\rm{Sp}}
(n,\R )$, we associate via  \eqref{pana}
 $g\mapsto g_{\C}\in \rm{Sp}(n,\R)_{\C}\equiv\rm{Sp}(n, \C)\cap
\rm{U}(n,n)$
 \begin{equation}\label{dgM}
g_{\C}= \left( \begin{array}{cc}p & q\\ \bar{q} &
\bar{p}\end{array}\right),
\end{equation}
where the matrices $p,q\in M(n,\C)$ have the properties
\begin{subequations}\label{simplectic}
\begin{eqnarray}
& pp^*- qq^* = \un,\quad  pq^t=qp^t; \label{simp1}\\
& p^*p-q^t\bar{q} = \un ,\quad p^t\bar{q}=q^*p .\label{simp2}
\end{eqnarray}
\end{subequations}

e) The characteristic polynomial of a real Hamiltonian
  matrix is an even polynomial. If $\lambda$ is an eigenvalue of a
  Hamiltonian matrix with multiplicity $k$, so are $-\lambda,
  \bar{\lambda}, -\bar{\lambda}$ with the same multiplicity.  Moreover, 0, if it occurs, has
even multiplicity.  If
  $A\in\got{sp}(n,\R)$ has distinct eigenvalues
  $\lambda_1,\dots,\lambda{_n},
-\lambda_1,\dots,-\lambda{_n}$, there exists a symplectic matrix $S$
(possibly complex) such as $S^{-1}AS={\emph{\text{diag}}}(  \lambda_1,\dots,\lambda_n,
-\lambda_1,\dots,-\lambda_n)$.

The characteristic polynomial of a symplectic matrix is a reciprocal
polynomial. If $\lambda$ is an eigenvalue of a real symplectic matrix
with multiplicity $k$, so are and
$\lambda^{-1}, \bar{\lambda}, \bar{\lambda}^{-1}$ with the same multiplicity. Moreover, the
multiplicities of the eigenvalues $+1$ and $-1$, if they occur, are
even.

\end{Remark}
\section{Coherent states on $\mc{D}^J_n$}\label{mare} 
In order to fix the  notation on coherent states \cite{perG}, let us consider the triplet $(G, \pi , \got{H} )$, where $\pi$ is
 a continuous, unitary 
representation
 of the  Lie group $G$
 on the   separable  complex  Hilbert space \Hi .

For $X\in\g$, where \g ~is the Lie algebra of the Lie group $G$,  let
us define the (unbounded) operator $\dd\pi(X)$ on \Hi~ by
$\dd\pi(X).v:=\left. {\dd}/{\dd t}\right|_{t=0} \pi(\exp tX).v$,
whenever the limit on the right hand side exists. 
   We obtain a
representation of the Lie algebra \g~ on \Hinf (\Hinf  denotes the
smooth vectors of \Hi), {\it the derived
representation}, and we denote
${\mb{X}}.v:=\dd\pi(X).v$ for $X\in\g ,v\in \Hinf$. 

Let us now denote by $H$  the isotropy group.  
We  consider (generalized) coherent
 states on complex  homogeneous manifolds $M\cong
G/H$ \cite{perG}. 
{\em The coherent vector
 mapping} is defined locally, on a coordinate neighborhood $\mc{V}_0$,
 $\varphi : M\rightarrow \bar{\Hi}, ~ \varphi(z)=e_{\bar{z}}$
(cf. \cite{last,sb6}),
where $ \bar{\Hi}$ denotes the Hilbert space conjugate to $\Hi$.
The  vectors $e_{\bar{z}}\in\bar{\Hi}$ indexed by the points
 $z \in M $ are called  {\it
Perelomov's coherent state vectors}. Explicitly, $e_z=\exp(\sum_{\alpha\in\Delta_+}z_{\alpha}\mb{X}_{\alpha})e_0$, where
$e_0$ is the extremal weight vector of the representation $\pi$, $\Delta_+$ are the positive roots of the Lie algebra $\got{g}$ of $G$,
and $X_\alpha,\alpha\in\Delta$,  are the  generators \cite{perG}. 

The space \fl~ of holomorphic functions 
is defined as the set of square integrable 
functions 
 with respect to  the scalar product
\begin{equation}\label{scf}
(f,g)_{\fl} =\int_{M}\bar{f}(z)g(z)\dd \nu_M(z,\bar{z}),
~~\dd{\nu}_{M}(z,\bar{z})=\frac{\Omega_M(z,\bar{z})}{(e_{\bar{z}},e_{\bar{z}})}.
\end{equation}
Here  $\Omega_M$ is the normalized  $G$-invariant volume form
\begin{equation}
\Omega_M:=(-1)^{\binom {n}{2}}\frac {1}{n!}\;
\underbrace{\omega\wedge\ldots\wedge\omega}_{\text{$n$ times}}\ ,
\end{equation}
and the $G$-invariant  K\"ahler two-form $\omega$ on the $2n$-dimensional
manifold  $M$ is given by
\begin{equation}\label{kall}
\omega(z)=\i\sum_{\alpha\in\Delta_+} \mc{G}_{\alpha,\beta}  \dd z_{\alpha}\wedge
\dd\bar{z}_{\beta}, 
\end{equation}
\begin{equation}\label{Ter} 
\mc{G}_{\alpha,\beta}=\frac{\pa^2}{\pa
  z_{\alpha} \pa\bar{z}_{\beta}} \log <e_{\bar{z}},e_{\bar{z}}>,
~\mc{G}_{\alpha,\beta}= \bar{\mc{G}}_{\beta,\alpha}.
\end{equation}

   (\ref{scf})
is nothing else but   {\em {Parseval   overcompletness
 identity}}  \cite{berezin}.
 
Let us  now introduce the map $\Phi :\Hi^{\star}\!\rightarrow \fl$ ,
\begin{equation}\label{aa}
\Phi(\psi):=f_{\psi},
f_{\psi}(z)=\Phi(\psi )(z)=(\varphi (z),\psi)_{\Hi}=(e_{\bar{z}},\psi)_{\Hi},
~z\in{\mathcal{V}}_0, ~\mathcal{V}_0\subset M, 
\end{equation}
where we have identified the space $\overline{\Hi}$  complex conjugate
 to \Hi~  with the dual
space
$\Hi^{\star}$ of $\Hi$.

Perelomov's coherent state vectors  \cite{perG}  associated to the group $G^J_n$ with 
Lie algebra the Jacobi algebra $\got{g}^J_n$, based on the complex
 $N$-dimensional
($N= \frac{n(n+3)}{2}$)  manifold - {\it the Siegel-Jacobi ball}
$ \mc{D}^J_n:= H_n/\R\times \text{Sp}(n,\R )_{\C}/\text{U}(n)=
\C^n\times\mc{D}_n $ -
are defined as \cite{sbj,perG}
\begin{equation}\label{csuX}
e_{z,W}= \exp ({\mb{X}})e_0, 
~\mb{X} := \sum_{i=1}^n z_i \mb{a}^{\dagger}_i + \sum_{i,j =1}^n w_{ij}\mb{K}^+_{ij},~
 z\in \C^n;    W\in\mc{D}_n.
\end{equation}

 The non-compact hermitian
symmetric space $ \operatorname{Sp}(n, \R
)_{\C}/\operatorname{U}(n)$ admits a matrix realization  as a bounded
homogeneous domain, the Siegel ball $\mc{D}_n$
\begin{equation}\label{dn}
\mc{D}_n:=\{W\in  M (n, \C ): W=W^t, \un-W\bar{W} > 0\}.
\end{equation}
$\mc{D}_n$ is a hermitian symmetric space of type CI (cf.~Table V, p.~
518, in \cite{helg}), identified with the  symmetric bounded
domain of type II, $\got{R}_{II}$ in Hua's notation \cite{hua}.
\enlargethispage{1cm}

The vector $e_0$  appearing in (\ref{csuX})   verifies the relations 
\begin{equation}\label{vacuma}
\mb{a}_ie_o= 0,~ i=1, \cdots, n; ~
\mb{K}^+_{ij} e_0  \not=  0 ,~
\mb{K}^-_{ij} e_0 = 0 ,~
\mb{K}^0_{ij} e_0 =  \frac{k_i}{4}\delta_{ij} e_0, ~i,j=1,\dots,n .
\end{equation}
In (\ref{vacuma}), $e_0=e^H_0\otimes e^K_0$, where $e^H_0$ is the
minimum weight vector (vacuum) for the Heisenberg group $H_n$ with
respect to the representation \eqref{deplasareM}, while $e^K_0$ is the
extremal weight vector for $\text{Sp}(n,\R)_{\C}$ corresponding to the
weight $k$ in (\ref{vacuma}) with respect to a unitary  representation
$S$,  see details in \cite{sbj}.

The following proposition describes the holomorphic action of the
Jacobi group 
on the Siegel-Jacobi ball and some geometric
properties of $\mathcal{D}^J_n$ (cf. \cite{sbj}): 
\begin{Proposition}\label{mm11}
Let us consider the action $S(g)D(\alpha )e_{z,W}$, where $g\in
\operatorname{Sp}(n,\R )_{\C}$ is of the form \eqref{dgM},
\eqref{simplectic}, $D(\alpha)$ is given by  \eqref{deplasareM}  and the coherent state vector is
defined in \eqref{csuX}. Then formulas \eqref{TOIU} hold:
\begin{subequations}\label{TOIU}
\begin{align}\label{xx}
S(g)D(\alpha )e_{z,W} & =\lambda e_{z_1,W_1}, \quad \lambda = \lambda
(g,\alpha; z,W) , \\
\label{x44} W_1 & =g\cdot
W= (pW+q)(\bar{q}W+\bar{p})^{-1}=(Wq^*+p^*)^{-1}(q^t+Wp^t),\\
\label{xxxX}
z_1 & = (Wq^*+ p^*)^{-1}(z+ \alpha -W\bar{\alpha}),\\
\label{x1} \lambda & = \det (Wq^*+p^*)^{-k/2} \exp
\left(\tfrac{\bar{\eta}}{2}z -\tfrac{\bar{\eta_1}}{2}{z_1}\right) \exp
(\i\theta_h(\alpha , \eta)),\\ \label{x2} \eta & =
(\un-W\bar{W})^{-1}(z+W\bar{z}),\\
\label{x3} \eta_1 & =p (\alpha + \eta ) + q (\bar{\alpha}+\bar{\eta}).
\end{align}
\end{subequations}
 The action of the Jacobi group $G^J_n$ on the
manifold  $\mathcal{D}^J_n$   is given by equations \eqref{xxxX},
\eqref{x44}. The composition law is
\begin{equation}\label{compositieX}
(g_1,\alpha_1,t_1)\circ (g_2,\alpha_2, t_2)= \big(g_1\circ g_2,
g_2^{-1}\cdot \alpha_1+\alpha_2, t_1+ t_2 +\Im
(g^{-1}_2\cdot\alpha_1\bar{\alpha}_2)\big),
\end{equation}
and   if $g$ is as in \eqref{dgM}, then $g\cdot\alpha :=\alpha_g$ is given by 
$\alpha_g = p\,\alpha +
q\,\bar{\alpha} $, and $g^{-1}\cdot\alpha ={p}^*\alpha
-q^t\bar{\alpha}$.

The
manifold $\mc{D}^J_n$ has the K\"ahler potential \eqref{kelerX},
$f=\log K$, with $K$ given by \eqref{kul}, 
\begin{equation}\label{kelerX}
\begin{split}
f = & -\tfrac{k}{2}\log \det (\un-W\bar{W})+
\bar{z}^t(\un-W\bar{W})^{-1}{z}\\
& + \tfrac{1}{2}z^t[\bar{W}(\un-W\bar{W})^{-1}]z
+\tfrac{1}{2}\bar{z}^t [ (\un-W\bar{W})W]\bar{z}.
\end{split}
\end{equation}
 The  K\"ahler two-form $\omega_n$, deduced as in \eqref{kall},  $G^J_n$-invariant to the action
\eqref{x44}, \eqref{xxxX}  is
\begin{equation}\label {aabX}
\begin{split}
- \i\omega_n &\! =\! \tfrac{k}{2}\tr (B\wedge\bar{B})\!
 +\!\tr (A^t\bar{M}\wedge \bar{A}),A =\dd z\!+\!\dd W\bar{\eta},\\
B & = M\dd W,    M ~\! =\!(\un-W\bar{W})^{-1},~\eta=M(z+W\bar{z}). 
\end{split}
\end{equation}

The scalar product $K:M\times
\bar{M}\rightarrow\C$,  $K(\bar{x},\bar{V};y,W)=(e_{x,V},e_{y,W})_k $ 
is:
\begin{equation}\label{KHKX}
\begin{split}
(e_{x,V},e_{y,W})_k  & = \det (U)^{k/2} \exp
F(\bar{x},\bar{V};y,W),~U= (\un-W\bar{V})^{-1}; \\
  2 F(\bar{x},\bar{V};y,W) & =  2\langle x, U y\rangle +\langle V\bar{y},U y\rangle +\langle x,U
W\bar{x}\rangle . 
\end{split}
\end{equation}
In particular,  the reproducing kernel $K=(e_{z,W},e_{z,W})$ is 
\begin{equation}\label{kul}
K=\det(M)^{\frac{k}{2}}\exp F, M=(\un-W\bar{W})^{-1}, 
\end{equation}
\begin{equation}\label{FfF}
2F=2\bar{z}^tMz+z^t\bar{W}Mz+\bar{z}^tMW\bar{z}.
\end{equation}
The Hilbert space of holomorphic functions $\mc{F}_K$
associated to the holomorphic kernel $K$ given by \eqref{kul} is endowed with the
scalar product of the type \eqref{scf}  
\begin{equation}\label{ofiX}
(\phi ,\psi )= \Lambda_n \int_{z\in\C^n;
\un -W\bar{W}>0}\bar{f}_{\phi}(z,W)f_{\psi}(z,W)Q K^{-1} \dd z \dd W,
\end{equation}
 where the normalization constant
$\Lambda_n$ is given by \eqref{ofi1X}
\begin{equation}\label{ofi1X}
\Lambda_n = \frac{k-3}{2\pi^{n(n+3)/2}}\prod_{i=1}^{n-1}
\frac{(\frac{k-3}{2}-n+i)\Gamma (k+i-2)}{\Gamma [k+2(i-n-1)]} .
\end{equation}
and the density of volume given is given  by \eqref{QQQ}
\begin{equation}\label{QQQ}
Q = \det (\un-W\bar{W})^{-(n+2)},
\dd z = \prod_{i=1}^n \dd \Re z_i \dd \Im z_i; ~~ \dd W = \prod_{1\le i\le j \le n}
\dd \Re w_{ij} \dd \Im {w}_{ij}. 
\end{equation}
Comparatively   with the case  of the
symplectic group, a shift of $p$ to $p- 1/2$ in the normalization
constant \eqref{ofiW}
$\Lambda_n = \pi^{-n}J^{-1}(p)$ is obtained.
We write down the scalar product \eqref{ofiX} also as ($ p=k/2-n-2$)
\begin{equation}\label{ofiXX}
(\phi ,\psi )\!= \!\Lambda_n \int_{z\in\C^n; W\in\mc{D}_n}\!\bar{f}_{\phi}(z,W)f_{\psi}(z,W)\rho_1 \dd z \dd W,\rho_1
\!=\!\det(\un\!-W\bar{W})^ p\exp{-F}.
\end{equation}
\end{Proposition}

\section{The differential action}\label{DIFFAC}

Let us consider again the triplet $(G, \pi, \Hi )$ introduced at the
beginning of \S \ref{mare}.
 The
unitarity and the continuity of  the representation $\pi$ 
imply that $\i \dd\pi (X)|_{\Hinf}$ is essentially selfadjoint.
Let us denote his image in $B_0(\Hinf )$ by $\am := \dd\pi(\Ugc)$, where
$B_0$ denotes the set of linear oparators with formal adjoint, and
$\Ugc$ denotes the universal covering algebra.
If $\Phi : \Hi^*\rightarrow \fl $ is the isometry
(\ref{aa}), we are interested in the study of the image of
\am~  via $\Phi$  as subset in the
algebra of holomorphic, linear differential operators,
$ \Phi\am\Phi^{-1}:={\db{A}}_M\subset\got{D}_M$.

The   set
 $\got{D}_M$ (or simply \D ) {\it of holomorphic, finite order, linear
differential operators on} $M$ is a
 subalgebra of homomorphisms ${\mc Hom}_{\C}({\mc O}_M,{\mc O}_M)$
 generated
 by the set ${\mc O}_M$ of germs of holomorphic functions of $M$ and the
 vector fields.
 We consider also {\it  the subalgebra} \AM~ of ${\db{A}}_M$~
 {\it of differential operators with  holomorphic polynomial coefficients}.
Let $U:=\mc{V}_0\subset M$, endowed with the local coordinates
$(z_1,z_2,\cdots ,z_n)$. We set $\pa_i :=\frac{\pa}{\pa z_i}$ and
$\pa^{\alpha}:=\pa^{\alpha_1}_1
\pa^{\alpha_2}_2\cdots \pa^{\alpha_n}_n$, $\alpha :=(\alpha_1,
\alpha_2 ,\cdots ,\alpha_n)\in\N^n$. The sections of \DM~ on $U$ are
$A:f\mapsto \sum_{\alpha}a_{\alpha}\pa^{\alpha}f$,
$a_{\alpha}\in\Gamma (U, {\mc{O}})$,  $a_{\alpha}$-s being zero
except a finite number.

For $k\in\N$, let us denote by $\D_k$ the subset of differential
operators of degree $\le k$.  The filtration of
$\D$ induces a filtration on $\got{A}$. 

Summarizing, we have a correspondence between the following three
objects \cite{last,sb6}:
\begin{equation}\label{correspond}
\g_{\C} \ni X \mapsto\mb{X}\in\am\mapsto\db{X}\in\AAA\subset \DM, {\mbox{\rm
 ~ {differential~operator~on}}}~ \fl .
\end{equation}

Moreover, it is easy to see \cite{last} that
 {\it if $\Phi$ is the isometry
{\em{(\ref{aa})}}, then
$\Phi d\pi(\g_{\C})\Phi^{-1}$ $\subseteq \D_1$ and we have}
\begin{equation}\label{car1}
\g_{\C}\ni X \mapsto{\db{X}}\in \D_1;~~
 {\db{X}}_z(f_{\psi}(z))=
\db{X}_z(e_{\bar{z}},\psi)= (e_{\bar{z}},\mb{X}\psi),
\end{equation}
where
\begin{equation}\label{sss}
\db{X}_{z}(f_{\psi}(z))=\left(P_{X}(z)+
\sum Q^i_X(z)\frac{\pa}{\pa  z_i}\right)f_{\psi}(z) .
\end{equation}
In \cite{last,sb6}  we have advanced the hypothesis that for {\it
  coherent state groups
the holomorphic functions $P$ and $Q$ in  {\em(\ref{sss})} are
polynomials},
 i.e. $\db{A} \subset \got{A}_1\subset \D_1$.

In particular, the Jacobi algebra $\got{g}^J_n$ admits a realization in the space
$\got{D}_1$ of  holomorphic first order
differential operators with polynomial coefficients, defined on the
Siegel-Jacobi  ball $\mc{D}^J_n$.  The space of holomorphic
functions on which the differential operators act is  the space
denoted $\mc{F}_K$  in
Proposition \ref{mm11}. For explicit realization of the
representation  \cite{sbcag,sbl,sb6},  use is made of  the formula   $\Ad(\exp X)=\exp(\ad_X)$, i.e.
\begin{equation}\label{adjk}
A\e^X=\e^X(A-[X,A]+\frac{1}{2}[X,[X,A]]+\dots), 
\end{equation}
  In order to take    into
account the symmetry of the matrix $W$ appearing in  \eqref{csuX}, we use the
derivation formula:
\begin{equation}\label{mircea}
\frac{\pa w_{ij}}{\pa w_{pq}} =
\delta_{ip}\delta_{jq}+\delta_{iq}\delta_{jp}-\delta_{ij}\delta_{pq}\delta_{ip},
~ w_{ij}=w_{ji} .
\end{equation}

With \eqref{adjk} and taking into account the commutation relations
\eqref{baza1M}, \eqref{baza3M} and  \eqref{baza2M}  of the generators
of the Jacobi algebra,   we get the relations \eqref{Ppark}  (see also \S 2.4, 3.3 in \cite{sbj}, where $\chi$ has been
omitted,  and  \cite{Q90}):
\begin{subequations}\label{Ppark}
\begin{align}
\mbox{~~~~~} \mb{a}^{\dagger}_ke_{z,W} & =  \frac{\pa}{\pa z_k}e_{z,W}, ~
\mb{a}_ke_{z,W}  =  \left( z_k + w_{ki}\frac{\pa}{\pa
    z_i}\right)e_{z,W} ; \\
\mbox{~~~~~}\mb{K}^0_{kl}e_{z,W}  & =   \left( \frac{k_k}{4}\delta_{kl}+
\frac{z_l}{2}\frac{\pa}{\pa z_k} +w_{li}\nabla_{ik}\right) e_{z,W},~  \mb{K}^+_{kl}e_{z,W}   =
\nabla_{kl}e_{z,W} ; \\
\mbox{ ~~~~~}\mb{K}^-_{kl}e_{z,W}  & =   \left[\frac{k_k+k_l}{4}w_{kl} + \frac{z_kz_l}{2}
+\frac{1}{2}(z_lw_{ik}+z_kw_{il})\frac{\pa}{\pa z_i}
+w_{\alpha l}w_{ki} \nabla_{i\alpha}\right]e_{z,W}. 
\end{align}
\end{subequations}

Now we briefly  recall some general considerations (for details see
\cite{last,jac1} and   Lemma 1
in \cite{sb6}). For
  $X\in\got{g}$, let $\mb{X}.e_{z}:=\mb{X}_ze_z$. Then
$\mb{X}.e_{\bar{z}}={\mb{X}}_{\bar{z}}e_{\bar{z}}$.

 But
$(e_{\bar{z}},
{\mb{X}}.e_{\bar{z}'})=({\mb{X}}^{\dagger}.e_{\bar{z}},e_{\bar{z}'})$ and
finally, with equation (\ref{car1}), we have
\begin{equation}\label{bidiff}
\db{X}_{\bar{z}'}(e_{\bar{z}},e_{\bar{z}'})=\db{X}^{\dagger}_{z}
(e_{\bar{z}},e_{\bar{z}'})
\end{equation}
Using \eqref{Ppark} and the relation expressed in \eqref{bidiff}, we have
\begin{lemma} \label{PERKK} The  Jacobi algebra  $\got{g}^J_n$ admits
  a realization in the space $\got{D}_1$ of  differential operators on $\mc{D}^J_n$:
 \begin{subequations}\label{park}
\begin{align}
\mbox{~~~~~} \mb{a}_k& =  \frac{\pa}{\pa z_k},~ ~
\mb{a}^{\dagger}_k  =   z_k + w_{ki}\frac{\pa}{\pa z_i} ; \label{nabla1}\\
\mbox{~~~~~}\mb{K}^0_{kl}& =    \frac{k_k}{4}\delta_{kl}+
\frac{z_k}{2}\frac{\pa}{\pa z_l} +w_{ki}\nabla_{il}, ~   \mb{K}^-_{kl}   =   \nabla_{kl}; \label{nabla2}\\
\mbox{ ~~~~~}\mb{K}^+_{kl} & =   \frac{k_k+k_l}{4}w_{kl} + \frac{z_kz_l}{2}
+\frac{1}{2}(z_lw_{ik}+z_kw_{il})\frac{\pa}{\pa z_i}
+w_{\alpha l}w_{ki} \nabla_{i\alpha}. \label{nabla3}
\end{align}
\end{subequations}
In the formulae above, $k,l=1,\dots,n$,$w_{kl}=w_{lk}$, $\frac{\pa}{\pa w_{kl}}=\frac{\pa}{\pa w_{kl}}$, and the dummy summation  is on all indexes $1,\dots,n$. 

With the convention $\nabla=(\nabla_{ij})_{i,j=1,\dots,n}=(\chi_{ij}\frac{\pa }{\pa
  w_{ij}})_{i,j=1,\dots,n}$, formulae \eqref{park} can be also written as
\begin{subequations}\label{difjac}
\begin{align}
\mbox{~~~~}\mb{a} & =  \frac{\pa}{\pa z},~ ~
\mb{a}^{\dagger}  =  z + W\frac{\pa}{\pa z}; \\
\db{K}^- & =  \nabla_ W, ~
\db{K}^0  =  \frac{k}{4} + \frac{1}{2}\frac{\pa}{\pa z}\otimes z 
+\nabla_WW; \\
\db{K}^+ & =   \frac{W'}{4}+\frac{1}{2}z\otimes z 
+\frac{1}{2}(W\frac{\pa}{\pa z}\otimes z+ z\otimes \frac{\pa}{\pa z}W )
+ W\nabla_WW .
\end{align}
\end{subequations}

In  \eqref{difjac}  $A\otimes B$ denotes the Kronecker product
 of matrices, here 
$(A\otimes B)_{kl}=a_kb_l$, $A=(a_k), ~B= (b_l)$,  $k=
{\emph{\text{diag}}}(k_1,\dots,k_n)$, $w'_{kl}=(k_k+k_l)w_{kl}$, $ k, l = 1,\dots , n$.
\end{lemma}

We particularize the values of the operators given in
Lemma \ref{PERKK} in the case of the action of  $\text{Sp}(n,\R)_{\C}$ on
$\mc{D}_n$  and we have
\begin{lemma}\label{lema2}
 The  algebra $\got{sp}(n,\R)_{\C}$  defined by the commutations
 relations \eqref{baza2M}   admits the realization
 in differential operators on the Siegel ball $\mc{D}_n$
\begin{equation}\label{Parkk}
\mb{K}^0_{kl} =    \frac{k_k}{4}\delta_{kl}
+w_{ki}\nabla_{il}, ~
 \mb{K}^-_{kl}   =   \nabla_{kl}, ~
\mb{K}^+_{kl}  =   \frac{k_k+k_l}{4}w_{kl} +
w_{\alpha l}w_{ki} \nabla_{i\alpha}. 
\end{equation}
The generators of $\emph{\text{Sp}}(n,\R)_{\C}$ have the hermiticity properties
\begin{equation}\label{HERMM}
(\mb{K}^+)^{\dagger}_{kl}=\mb{K}^-_{lk}=\mb{K}^-_{kl} ,
(\mb{K}^0)^{\dagger}_{kl}=\mb{K}^0_{lk}
\end{equation} with
respect to the scalar product (\cite{sbj})
\begin{equation}\label{ofiW}
(\phi ,\psi )_{\fl}= \Lambda'_n \int_{1-W\bar{W}>0}
\bar{f}_{\phi}(W)f_{\psi}(W) \rho \dd W, ~
\end{equation}
where  $k_i=k$, $i=1,\dots,n$ in \eqref{vacuma}, 
\begin{equation}\label{rhoo}
\rho = \det (\un-W\bar{W})^q , ~
q=\tfrac{k}{2}- n-1
\end{equation} and 
$\Lambda'_n =J_n^{-1}(q)$, with $J_n(p)$ given by \eqref{JJJ1},
$p>-1$,
\begin{equation}\label{JJJ1}
\!\!J_n(p)\!=\!2^n\pi^{\frac{n(n+1)}{2}}\!\prod_{i=1}^n\!\frac{\Gamma
(2p+2i)}{\Gamma (2p+n+i+1)}\!=\!\frac{\pi^{n(n+1/2)}}{p+n}\!\prod_{i=1}^{n-1}\frac{\Gamma[2(i+p-1)]}{(i+p)\Gamma[i+2(p+n-1)]}.
\end{equation}
\end{lemma}
\begin{proof}
Firstly we have to check up  that the operators (\ref{Parkk}) verify the
commutation relations (\ref{baza2M}). This is a long calculation,
based on the   formula \eqref{mircea}.

Then  we verify the relation
$(\mb{K}^+_{kl}f,g)_{\fl}=(f,\mb{K}^-_{lk}g)_{\fl}$,  which imposes to
the reproducing kernel $\rho$ the condition
\begin{equation}\label{COND1}
(\nabla_{kl}
+q\bar{w}_{kl}-\bar{w}_{ki}\bar{w}_{lj}\bar{\nabla}_{ij})\rho= 0.
\end{equation}
The hermiticity condition
$(\mb{K}^0)^{\dagger}_{kl}=\mb{K}^0_{lk}$  of the
operators given by the first formula (\ref{Parkk}) with respect to the scalar
product (\ref{ofiW}),  with $\rho$ given by (\ref{rhoo}), imposes to the
kernel function $\rho$ the condition:
\begin{equation}\label{COND2}
(\bar{w}_{kp}\bar{\nabla}_{lp}-
w_{lp}\nabla_{kp})\rho = 0.
\end{equation}
The conditions (\ref{COND1}) and (\ref{COND2}) are verified by the
kernel (\ref{rhoo}) using 
the relation
\begin{equation}\label{derA}
\frac{\pa }{\pa w_{ik}}  A  =
 (-2X_{ik}+X_{ik}\delta_{ik}) A, ~\text{or}~ \nabla A= -X, ~~\text{where}~ A=\det(\un-W\bar{W}), 
\end{equation}
which implies
\begin{equation}\label{COCO}
\frac{\pa \rho}{\pa {w}_{ab}}   =  
q (-2X_{ab}+X_{ab}\delta_{ab}) \rho, ~\text{or}~ \nabla \rho =
-qX\rho, ~\text{where} ~X =X^t= \bar{W}(\un-W\bar{W})^{-1}.
\end{equation}
Indeed, with \eqref{COCO}, the condition \eqref{COND1} reads
$-X+\bar{W}+\bar{W}\bar{X}\bar{W} =0$, while \eqref{COND2} reads
$(\bar{X}\bar{W})^t=XW$. The last two conditions are verified because
of the symmetry of the matrices $X$ and $W$. 
\end{proof}

\begin{lemma}\label{LEMMA4}
The pairs of operators $\mb{a}^{\dagger}$ and $\mb{a}$,
  $\mb{K}^+_{kl}$ and $\mb{K}^-_{kl}$,  $\mb{K}^0_{kl}$ and
  $\mb{K}^0_{lk}$  are respectively hermitian conjugate with respect
  the scalar product \eqref{ofiXX} for $k_i=k$ in \eqref{vacuma}.
\end{lemma}
\begin{proof}
We take the derivative of \eqref{FfF}  with respect with $z_i$ and we find successively
\[
\begin{split}
\frac{\pa F}{\pa z_i}  & = \bar{z}_pM_{pi}
+\frac{1}{2}[(\bar{W}M)_{iq}z_q+z_{p}(\bar{W}M)_{pi}]\\
~~~ & = [M^t\bar{z}+\frac{1}{2}(\bar{W}Mz+\bar{M}\bar{W}z)]_{i}\\
~~~ & = [M^t(\bar{z}+\bar{W}z)]_i, 
 \end{split}
\] 
and we introduce  $ \eta  =  M(z+W\bar{z}) , ~ M=(\un-W\bar{W})^{-1}$. We get 
\begin{equation}
\label{DERCV}
\frac{\pa F}{\pa z_i} =  \bar{\eta}_i,
\end{equation}
\begin{equation}\label{DERCV2}
\frac{1}{\rho_1} \frac{\pa \rho_1 }{\pa z_i} =  -\bar{\eta}_i .
\end{equation}

We look for the derivative of  $\rho_1$ from \eqref{ofiXX} with
respect to the $w_{ik}$, and  we have
\begin{equation}\label{detA}
\frac{\pa \rho_1}{\pa w_{ik}}=(p \frac{\pa  A}{\pa
  w_{ik}} - A\frac{\pa F}{\pa w_{ik}})A^{p-1}\exp(-F),~\text{where}~ 
A=\det(\un-W\bar{W}), 
\end{equation}
and 
\begin{equation}\label{ultiM}
\frac{\pa F}{\pa w_{ik}}   =   \bar{\eta}_i\bar{\eta}_k- 
\frac{1}{2}\bar{\eta}_i\bar{\eta}_k\delta_{ik} ,\quad\text{or}
~\nabla F= \frac{1}{2}\bar{\eta}\otimes\bar{\eta} .
\end{equation}

Indeed, 
with formula (\ref{mircea}), we
get
\begin{subequations}\label{LIKK}
\begin{align}
\frac{\pa M_{ab} }{\pa w_{ik}} & =  M_{ai}X_{kb}+
M_{ak}X_{ib}- M_{ai}X_{ib}\delta_{ik},\label{LIK}\\
\frac{\pa X_{ab} }{\pa w_{ik}}  & = 
X_{ai}X_{bk}+X_{ak}X_{ib}-X_{ai}X_{ib}\delta_{ik} , \\
\frac{\pa \bar{ X}_{ab} }{\pa w_{ik}}  &  =  
M_{ai}M_{bk}+M_{ak}M_{bi}-M_{ai}M_{bk}\delta_{ik} \label{LIK2}.
\end{align}
\end{subequations}
In order to prove \eqref{LIK2}, we use the fact that $\bar{X}=MW$,
then we use \eqref{LIK}, \eqref{mircea}, and the formula
$\un+XW=\bar{M}=M^t$. 

We write \eqref{FfF} as
$$2F=2\bar{z}^tMz+z^tXz+\bar{z}^t\bar{X}\bar{z}$$
and,  with  \eqref{LIKK}, we get
\begin{align*}
2\frac{\pa F}{\pa w_{ik}} & =
2[(\bar{z}^tM)_i(Xz)_k+(\bar{z}^tM)_k(Xz)_i-(\bar{z}^tM)_i(Xz)_k\delta_{ik}
]+\\
~~~~~~~~~~~~~~~~+ & (z^tX)_i(Xz)_k +  (z^tX)_k(Xz)_i -
(z^tX)_i(Xz)_k\delta_{ik} +\\
~~~~~~~~~~~~~~~~+ & (\bar{z}^tM) _i(\bar{z}^tM)_k+(\bar{z}^tM)_k (\bar{z}^tM)_i-(\bar{z}^tM)_i (\bar{z}^tM)_k\delta_{ik}.
\end{align*}
Then we use twice the relations 
$(\bar{z}^tM+z^tX)^t=\bar{\eta}$,  and \eqref{ultiM}  get proved.

With \eqref{derA} and \eqref{ultiM}, we have for \eqref{detA} the
expression
\begin{equation}\label{DERCV1}
-\frac{1}{\rho_1}\frac{\pa\rho_1}{\pa w_{ik}} = 
p(2X_{ik}-X_{ik}\delta_{ik})+  \bar{\eta}_i\bar{\eta}_k- 
\frac{1}{2}\bar{\eta}_i\bar{\eta}_k\delta_{ik} , ~~\text{and}
-\frac{1}{\rho_1}\nabla \rho_1= pX
+\frac{1}{2}\bar{\eta}\otimes\bar{\eta}. 
\end{equation}

We have to verify $(\mb{a}_kf,g)=(f,\mb{a}_k^{\dagger}g)$  with
respect to the scalar product \eqref{ofiXX}, i.e.
$$\frac{\pa F}{\pa \bar{z}_k}=z_{k}+w_{ki}\frac{\pa F}{\pa z_i}. $$  
With \eqref{DERCV2}, the last condition reads
$\eta = z+W\bar{\eta}$,  which is true.

In order to verify
$(\mb{K}^0_{kl}f,g)=(f,\mb{K}^0_{lk}g)$ for the case $k_i=k$ in
formula \eqref{vacuma} with respect to the scalar product \eqref{ofiXX}, we use the differential
action for $\mb{K}^0_{kl}$ 
 in Lemma  \ref{PERKK}. If we denote the integrant in
 the second term by
  $f_{kl}$, using \eqref{ultiM}, \eqref{DERCV1}, \eqref{mircea} and the
  formula $z=\eta-W\bar{\eta}$, 
we
find
 $$\frac{f_{kl}}{\rho_1}=\frac{p}{2}\delta_{kl}+\frac{1}{2}\eta_l\bar{\eta}_k + p(W\bar{M}\bar{W})_{lk}
, $$ and $f_{kl}=\bar{f}_{lk}$, because the symmetry of the matrices
$W$ and $X$. 

We also find for the integrant of
$(\mb{K}^-_{kl}f,g)=(f,\mb{K}^+_{kl}g)$ the common value
$[p\bar{X}_{kl}+\frac{1}{2}\eta_l\eta_k]\rho_1$. 
\end{proof}

\section{The real Jacobi group $G^J_n(\R )$}\label{RELGRJ}

We consider the real Jacobi group $G^J_n(\R )=\text{Sp}(n,\R
)\ltimes H_n$, where $H_n$ is now the real Heisenberg group
of real dimension $(2n+1)$. Let $g=(M,X,k), g'=(M',X',k')\in G^J_n(\R )$, where
$X=(\lambda,\mu)\in\R^{2n}$ and $(X,k)\in {H}_n$. Then the
composition law in $G^J_n(\R )$ is 
\begin{equation}\label{clawr}
gg'=(MM',XM'+ X', k+k'+XM'JX'^t).
\end{equation}
We shall also consider the restricted real Jacobi group $G^J_n(\R
)_0$, consisting only of elements of the form above, but $g=(M,X)$.

We consider also the Siegel-Jacobi upper half plane
$\mc{X}^J_n:= \mc{X}_n\times\R^{2n},$  where $\mc{X}_n=\text{Sp}(n,\R)/\text{U}(n)$ is  Siegel
upper half plane  realized as $$\mc{X}_n:=\{v\in M(n,\C)| v=s+\i  r, s,
r\in M(n,\R), 
r>0,
 s^t=s; r^t=r\} . $$

Let us consider an element $h=(g,l)$ in $G^J_n(\R )_0$, i.e.
\begin{equation}\label{mM}
g=\left(\begin{array}{cc} a & b\\ c &
d\end{array}\right)\in{\text{Sp}}(n,\R),
~l=(n,m)\in\R^{2n},  \end{equation}
and $v\in\mc{X}_n,~u\in\C^n\equiv\R^{2n}$. 

Now we consider  the partial Cayley transform \cite{sbj}
$\Phi:\mc{X}^J_n\rightarrow \mc{D}_n^J , ~\Phi(v,u)=(W,z)$
\begin{subequations}\label{bigtransf}
\begin{eqnarray}W & = & (v-\i\un )(v+\i\un )^{-1},  \label{big1}\\
 z & = & 2\i (v+\i\un)^{-1} u,\label{big2}
\end{eqnarray}
\end{subequations}
with the inverse partial Cayley transform
$\Phi^{-1}:\mc{D}^J_n\rightarrow \mc{X}^J_n$, $\Phi^{-1}(W,z)=(v,u)$
\begin{subequations}\label{bigtransF}
\begin{eqnarray} v & = & \i (\un-W ) ^{-1} (\un+W ),  \label{big11}\\
 u & = & (\un-W)^{-1}z. \label{big22}
\end{eqnarray}
\end{subequations}

Let us now define $\Theta: G^J_n(\R)_0\rightarrow G^J_n$, $\Theta(h)=h_*$,
$h=(g,n,m)$, $h_*=(g_{\C},\alpha)$.  
We shall verify that (see also \cite{Y08,gem})
\begin{Proposition}\label{THETAS}
$\Theta$ is an group isomorphism and the action of $G^J_n$ on $\mc{D}^J_n$ is compatible
  with the action of $G^J_n(\R )_0$ on  $\mc{X}^J_n$ through the
  biholomorphic partial Cayley transform \eqref{bigtransf}, i.e. if
  $\Theta(h)=h_*$, then $\Phi h=h_*\Phi$. More exactly,  if the action of $G^J_n$ on $\mc{D}^J_n$ is given by
  \eqref{x44}, \eqref{xxxX}, then the action of $G^J_n(\R)_0$ on
  $\mc{X}^J_n$ is given by $(g,l)\times(v,u)\rightarrow (v_1,u_1)\in\mc{X}^J_n$,
where
\begin{subequations}\label{conf}
\begin{align}\label{conf1}
v_1 & =  (av+b)(cv+d)^{-1} =(vc^t+d^t)^{-1}(va^t+b^t); \\
u_1& =  (vc^t+d^t)^{-1}(u+vn+m).\label{conf2}
\end{align}
\end{subequations}
The
  matrices $g$ in \eqref{mM} and $g_{\C}$  in \eqref{dgM} are related by
\eqref{pana}, \eqref{CUCURUCU}, while $\alpha=m+\i n$, $m,n\in\R^n$. 
\end{Proposition}
\begin{proof}
We introduce for $W$ in  \eqref{x44} its expression from \eqref{big1}, and we
get $$W_1=[(p+q)v+\i(q-p)][(\bar{q}+\bar{p})v+\i(\bar{p}-\bar{q})]^{-1}.$$
The expression of $W_1$ is introduced in the inverse \eqref{big11}  of
\eqref{big1}   for $v_1$. Use is made of 
 \eqref{CUCURUCU} and the first equality in \eqref{conf1} is
 obtained. The second one is a consequence of the symmetry of $v$, and it
 can be directly checked up with equations \eqref{lica}. 

For the second assertion, we start with \eqref{big2},  $2\i u_1=(v_1+\i\un)z_1$, in
which we introduce the expression \eqref{xxxX} for $z_1$. But with \eqref{big11}, 
$v_1+\i\un =2\i(\un-W_1)^{-1}$, so we get
$$u_1=(\un-W_1)^{-1}(Wq^*+p^*)[2\i(v+\i\un)^{-1}u+\alpha-W\bar{\alpha}].$$ 
In the above expression we write $W_1$ as function of $W$ with the
linear fractional transform \eqref{x44} and express $W$ as function of
$v$ with the  Cayley transform \eqref{big1}.  We replace $\alpha = m+\i n$,
$m,n\in\R^n$,  and
express the matrix elements  of $g_{\C}$ in function of the matrix elements
of $g$ via the relations \eqref{CUCURUCU} and we get also formula \eqref{conf2}.  
\end{proof}
\begin{Proposition}\label{PARTC}The partial Cayley transform  is a  K\"ahler
  homogeneous diffeomorfism,  $\Phi^*\omega_n=\omega'_n=\omega_n\circ\Phi$,  i.e. under the  transform
\eqref{bigtransf}, 
the K\"ahler two-form \eqref{aabX}  on $\mc{D}^J_n$, $G^J_n$-invariant
under the action \eqref{x44}, \eqref{xxxX},   becomes the K\"ahler
two-form $\omega'_n$ \eqref{kl2} on $\mc{X}^J_n$, $G^J_n(\R )_0$-invariant to the
action \eqref{conf}  
\begin{equation}\label{kl2}
\begin{split}
- \i \omega'_n & =\frac{k}{2}\tr (H \wedge\bar{H})+\frac{2}{\i}\tr
(G^tD\wedge\bar{G}), \quad{\emph{\text{where}}}\\
D & =  (\bar{v}-v)^{-1}, H = D\dd v ;~ G =\dd u-\dd vD(\bar{u}-u).
\end{split}
\end{equation}
\end{Proposition}
\begin{proof}The expression \eqref{kl2} was deduced firstly in
  \cite{mlad}. Here we shortly indicate how to check up   the group invariance. We
  calculate only the second term in the sum \eqref{kl2}  because the first one is
  well known. 

Differentiating \eqref{conf2}, we get
$$\dd u_1
=(vc^t+d^t)^{-1}[\dd u+\dd v(cv+d)^{-1}(dn-c(u+m))].$$
Now we calculate $\Psi=(v_1-\bar{v}_1)^{-1}(u_1-\bar{u}_1)$. 
Starting from \eqref{conf1} and taking into account \eqref{lica}, we have
$$v_1-\bar{v}_1= (\bar{v}c^t+d^t)^{-1}(v-\bar{v})(cv+d)^{-1}.$$
Using $$(\bar{v}c^t+d^t)(vc^t+d^t)^{-1}-\un =(\bar{v}-v)c^t(vc^t+d^t)^{-1}, $$
$$(\bar{v}c^t+d^t)(vc^t+d^t)^{-1}v-\bar{v}=(v-\bar{v})[\un+c^t(d^t)^{-1}v]^{-1},$$
we
get $$\Psi=dn-cm+(cv+d)(v-\bar{v})^{-1}[(\bar{v}c^t+d^t)(vc^t+d^t)^{-1}u-\bar{u}]. $$
 Taking into account \eqref{lica1} in the differential of
 \eqref{conf2}, we get $$\dd v_1=(vc^t+d^t)\dd v(cv+d)^{-1},$$
and we find
$$G_1=(vc^t+d^t)^{-1}(\dd u+\dd v\Xi),$$
where $$\Xi =-(v-\bar{v})^{-1}Y+(v-\bar{v})^{-1}\bar{u},$$
$$Y=(v-\bar{v})(cv+d)^{-1}c+(\bar{v}c^t+d^t)(vc^t+d^t)^{-1}.$$
Using relations of the type $v(cv+d)^{-1}c=
(\un+vd^{-1}c)^{-1}vd^{-1}c$, it can be shown that $Y=\un$, and we get $G_1=(vc^t+d^t)^{-1}G$, and
the invariance of the second term in formula \eqref{kl2} is proved. Then the
invariance of $\omega'_n$ under the  action of 
\eqref{conf} follows.  
\end{proof}
$\omega'_n$ given by (\ref{kl2}) is the  ``$n$''-dimensional generalization of
 Berndt-K\"ahler two-form $\omega'_1$.  In \S 37 in \cite{cal3}
 K\"ahler calls $\mc{X}^J_1$ {\it Phasenraum der Materie}, $v$ is {\it
   Pneuma}, $u$ is {\it Soma}. 
\subsection{Comparison with Yang's results}\label{YANGcomp}
\enlargethispage{1cm}
J.-H. Yang \cite{Y07}-\cite{Y10} considers  the Siegel-Jacobi space of
degree $n$ and order $m$, 
$\mathbf{H}_{n,m}=\mc{X}_n\times \C^{mn}$, 
 the  Heisenberg group
$$H^{(n,m)}_{\R}=\left\{(\lambda,\mu, \kappa)|
  \lambda,\mu\in M_{mn}(\R),\kappa \in M_{m}(\R)\right\}, $$
and  the Jacobi group $G^J=\text{Sp}(n,\R)\ltimes H^{(n,m)}_{\R}$,  with the multiplication law
$$(M_0,(\lambda_0,\mu_0,\kappa_0))\cdot (M,(\lambda,\mu,\kappa))=
(M_0M,(\tilde{\lambda}_0+\lambda,\tilde{\mu}_0+\mu,\kappa_0+\kappa+\tilde{\lambda}_0\mu^t-\tilde{\mu}_0\lambda^t)),$$
where $(\tilde{\lambda}_0,\tilde{\mu}_0)=(\lambda_0,\mu_0)M$.  $G^J$
acts transitively on the Siegel-Jacobi space $\mathbf{H}_{n,m}$ by 
$$(M,(\lambda,\mu,\kappa))(\Omega,Z)=(M\circ\Omega,(Z+\lambda\Omega+\mu)(C\Omega+D)^{-1}),$$
and  $G^J/K^J\cong\mathbf{H}_{n,m}$ is a nonreductive complex manifold, where $K^J=\text{U}(n)\times\text{Sym}(n,\R). $
Now we  identify  $(v,u)$ in our partial Cayley transform (\ref{conf}) with Yang's  partial Cayley transform  (11) in  \cite{Yan},
\begin{equation}\label{parY}
\Omega=\i (\un+W)(\un-W)^{-1} ; Z=2\i \eta (\un-W)^{-1};
\end{equation}
\begin{equation}\label{schimbB}
(v,u)\leftrightarrow(\Omega,\frac{Z^t}{2\i}) \quad\text{and\quad}  (W,z)
\leftrightarrow (W,\eta^t).
\end{equation}
\begin{Remark}\label{Rem2}
The K\"ahler two-form in the  case $m=1$ in {\emph{Theorem 1}} in  \cite{Yan}  is
the  K\"ahler two-form on $\mc{X}^J_n$ \eqref{kl2}, while the  K\"ahler
two-form on $\mc{D}^J_n$
\eqref{aabX}  is the corresponding one given in {\emph{Theorem 6}} in
 \cite{Yan}.
\end{Remark}
\begin{proof} 
We use also Yang's notation $\Omega= X +\i Y; Z= U+\i V$ and we
express  our (\ref{kl2})  in Yang's notation as 
$$ D= -\frac{1}{2\i}Y^{-1};~H=(-2\i Y)^{-1} \dd\Omega; ~G=\frac{\dd Z^t}{2\i}-\dd\Omega Y^{-1}V^t,$$
$$-\i w'_n=\frac{k}{8}\tr (Y^{-1}\dd\Omega\wedge Y^{-1}\dd\bar{\Omega})+
\frac{1}{8}\tr[(\dd Z-VY^{-1}\dd \Omega)Y^{-1}\wedge(\dd \bar{Z}^t-\dd
\bar{\Omega}Y^{-1}V^t)],$$
\begin{align*}
-\i w'_n & =  \frac{k}{8}\tr(Y^{-1}\dd \Omega\wedge Y^{-1}\dd \bar{\Omega})
+\frac{1}{8}\tr(\dd ZY^{-1}\wedge \dd\bar{Z}^t)
+\frac{1}{8}\tr(VY^{-1}\dd\Omega Y^{-1}\wedge \dd\bar{\Omega}Y^{-1}V^t)\\
 ~~ & ~ 
-\frac{1}{8}\tr(\dd ZY^{-1}\wedge \dd \bar{\Omega}Y^{-1}V^t)
-\frac{1}{8}\tr(VY^{-1}\dd \Omega Y^{-1}\wedge \dd \bar{Z}^t).
\end{align*}
The second term in the sum above reads
$$\dd Z_{ij}Y^{-1}_{jk}\wedge \dd\bar{Z}^t_{ki }= Y^{-1}_{kj}\dd Z_{ji}\wedge \dd\bar{Z}_{ik}= \tr(Y^{-1}\dd Z^t\wedge \dd\bar{Z}),$$
i.e. the corresponding term in Yang's formula and similarly for the other 3 terms in the sum. 
\end{proof}

\section{The fundamental conjecture for the Siegel-Jacobi domains}\label{FCNN}

Let us remind the {\it fundamental conjecture for homogeneous K\"ahler
manifolds}
 (Gin\-dikin -Vinberg): {\it  every homogenous K\"ahler manifold is a holomorphic fiber bundle over a homogenous bounded domain in which the fiber is the product of a flat homogenous K\"ahler manifold and a compact simply connected homogenous K\"ahler manifold}. The compact case was considered by Wang
\cite{wa}; Borel \cite{bo} and   Matsushima \cite{ma} have  considered the
case of a transitive reductive group  of automorphisms,  while Gindikin and Vinberg  \cite{GV}
considered a  transitive  automorphism group. We mention also the
essential contribution of Piatetski-Shapiro in this field
\cite{pia}. The complex version, in the formulation of Dorfmeister and
Nakajima \cite{DN}, essentially asserts that: {\it every homogenous
  K\"ahler manifold, as a complex manifold, is the product of a
  compact simply connected homogenous manifold (generalized flag
  manifold), a homogenous bounded domain,  and $\C^n/\Gamma$, where $\Gamma$ denotes a discrete subgroup of translations of $\C^n$}.

\begin{Proposition}\label{LKJ}
Under the homogeneous K\"ahler  transform $FC$ 
  \eqref{zzZ},
\begin{equation}\label{zzZ}\C^n\times\mc{D}_n\ni (\eta,W)\xrightarrow{FC} (z,W)\in\mc{D}^J_n,
FC(\eta,W)=(z,W), 
 ~ z=\eta - W\bar{\eta},
\end{equation} 
\begin{equation}\label{invFC}FC^{-1}: \eta =
  (\un-W\bar{W})^{-1}(z+W\bar{z}).
\end{equation}
the K\"ahler two-form  \eqref{aabX} on $\mc{D}^J_n$, 
  $G^J_n$-invariant to the action \eqref{x44}, \eqref{xxxX}, 
becomes the K\"ahler two-form on $\mc{D}_n\times\C^n$,
$FC^*\omega_n=\omega_{n,0}$, 
\begin{equation}\label {bX}
- \i\omega _{n,0}  = \tfrac{k}{2}\tr (B\wedge\bar{B})
 +\tr (\dd\eta^t\wedge \dd\bar{\eta}),
\end{equation}
invariant to the $G^J_n$-action on $\mc{D}_n\times\C^n$,
$(g,\alpha)\cdot (\eta,W)\rightarrow(\eta_1,W_1)$, with $W_1$ given in
\eqref{xxxX} and 
\begin{equation}\label{nouinv}
\eta_1=p(\eta+\alpha)+q(\bar{\eta}+\bar{\alpha}).
\end{equation}
Under the homogenous K\"ahler transform 
\begin{equation}\label{XCX1}
FC^{-1}_1:
\eta=(\bar{v}-\i\un)(\bar{v}-v)^{-1}(v-\i\un)[(v-\i\un)^{-1}u-(\bar{v}-\i\un)^{-1}\bar{u}].
\end{equation} 
the K\"ahler two-form \emph{(\ref{kl2})} becomes a K\"ahler two-form
on $\mc{X}_n\times\C^n$, $FC^*_1\omega'_n=\omega'_{n,0}$, 
\begin{equation}\label {bX1}
- \i\omega_{n,0}'   = \tfrac{k}{2}\tr (H\wedge\bar{H})
 +\tr (\dd\eta^t\wedge \dd\bar{\eta}), ~H=(\bar{v}-v) ^{-1}\dd v.
\end{equation}
The inverse transform of \eqref{XCX1} is
\begin{equation}\label{XCX2}
FC_1: u=\frac{1}{2\i}[(v+\i\un)\eta-(v-\i\un)\bar{\eta}].
\end{equation}
The K\"ahler two-form  \eqref{bX1} is invariant to the action
$G^J_n(\R)_0$ on $\mc{X}_n\times \C^n $, $(g,\alpha)\times
(v,\eta)\rightarrow(v_1,\eta_1)$, where $g$ has the form \eqref{mM},
$v_1$ is given by \eqref{conf1}, while 
\begin{equation}\label{cvr}
\eta_1 =\frac{1}{2}(\eta+\alpha)[a+d+\i(b-c)]+\frac{1}{2}(\bar{\eta}+\bar{\alpha})[a-d-\i(b+c)].
\end{equation}
\end{Proposition}
\begin{proof}
Following \cite{cal3}, \cite{ez}, we introduce the variables $P,Q\in\R^n$ such that
$u=vP+ Q$, where
$(u,v)\in\C^n\times\mc{X}_n$ are local coordinates on the Siegel-Jacobi
upper-half plane $\mc{X}^J_n$.  Using \eqref{big22}, we have $$ u
=vP+Q= (\un-W\bar{W})^{-1}z$$ and we introduce in formula above for $v$ the
expression given by \eqref{big11}. We get $z=\eta-W\bar{\eta}$, where
$\eta= P+\i Q$   has appeared already  in \eqref{x2} and 
\eqref{aabX}. For $A$ in \eqref{aabX}
we get  $A=\dd \eta -W \dd\bar{\eta}$.  In \eqref{aabX}, we make the
transform \eqref{zzZ}. 
Also, from (\ref{zzZ}) and (\ref{big1}), we have \eqref{x2},  i.e. \eqref{invFC}.

We use the relation $M-W\bar{M}\bar{W}=\un$ for the terms of the type
$\dd\eta^t_i\wedge \dd\bar{\eta}_j$,  the symmetry of the
matrices $\bar{M}\bar{W}$  for the terms of the type
$\dd\bar{\eta}^t_i\wedge \dd\bar{\eta}_j$ and $W\bar{M}$ for the terms of the
type $\dd\eta^t_i\wedge \dd\eta_j$, and  we  get for $\omega_{n,0}=\omega_n\circ FC$
the  expression given by \eqref{bX}.

Now we calculate the action of $G^J_n$ on $\C^n\times\mc{D}_n$ induced
from the action \eqref{x44}, \eqref{xxxX} on $\mc{D}^J_n$ applying the $FC$
transform \eqref{zzZ}.  We want to find $\eta_1$ from \eqref{invFC}
with $(W_1,z_1)$  given by \eqref{x44}, \eqref{xxxX},
\begin{equation}\label{Mi1}
(\un-W_1\bar{W}_1)\eta_1=z_1+W_1\bar{z}_1 .
\end{equation}
Firstly, with \eqref{simplectic}, we calculate the lhs of \eqref{Mi1} as
\begin{equation}\label{Mi2}
\un-W_1\bar{W}_1=(Wq^*+p^*)^{-1}(\un-W\bar{W})(q\bar{W}+p)^{-1},
\end{equation} 
where $p,q$ are components of the matrix $g\in\text{Sp}(n,\R)_{\C}$ as in \eqref{dgM}.

Now we  introduce the value $$z_1=(Wq^*+p^*)^{-1}[\eta+\alpha-W(\bar{\eta}+\bar{\alpha})],$$
  in  rhs of \eqref{Mi1}, and we write it as 
\begin{equation}\label{Mi3}
\begin{split}
z_1+W_1\bar{z}_1 & = (Wq^*+p^*)^{-1}Y, \quad{\text{where}}\\
Y & =\mc{P}(\eta+\alpha)+\mc{Q}(\bar{\eta}+\bar{\alpha}),\\
\mc{P} & =\un-(q^t+Wp^t)(\bar{W}b^t+a^t)^{-1}\bar{W}, \\
\mc{Q} & = -W +(q^t+Wp^t)(\bar{W}b^t+a^t)^{-1} . 
\end{split}
\end{equation}
Using  the successively the second relation in \eqref{simp1}, after an easy but long calculation, we find 
\begin{equation}\label{Mi4}
\begin{split}
\mc{P} & = (\un-W\bar{W})(p+q\bar{W})p ,\\
\mc{Q} & =(\un-W\bar{W})(p+q\bar{W})q.
\end{split}
\end{equation}
Combining \eqref{Mi3}-\eqref{Mi4}, we get for $\eta_1$ the value given
in \eqref{nouinv}. 

In order to verify the invariance of the K\"ahler two-form \eqref{bX}
to the action \eqref{nouinv}, we have to check up that 
\begin{equation}\label{trtr}
\tr(\dd\eta_1^t\wedge \dd\bar{\eta}_1) = \tr (\dd\eta^t\wedge
\dd\bar{\eta}),
\end{equation}
which is true because of the first relation in \eqref{simp2}.

With \eqref{big11}, we get
\begin{equation}\label{deltaw}
(\un-W\bar{W})^{-1}=\frac{1}{2\i}(\bar{v}-\i\un)(\bar{v}-v)^{-1}(v+\i\un).
\end{equation}
We introduce \eqref{deltaw}, \eqref{big2} and  \eqref{big1}  in
\eqref{x2} and we get \eqref{XCX1}.    

\eqref{XCX2} is obtained  introducing in     \eqref{big2}  the
expression \eqref{zzZ} with $W$ given by \eqref{big1}.

Finally, \eqref{cvr} is a consequence of \eqref{nouinv} and
\eqref{CUCURUCU1}.  Alternatively, the invariance \eqref{trtr},  where
$\eta_1$ has the expression \eqref{cvr},  can be checked up directly,  taking
into account that  the matrices $a,b,c,d$ appearing in \eqref{cvr} are the
components of $g\in\text{Sp}(n,\R)$ as in \eqref{mM},  and consequently
verify the conditions \eqref{lica}. 
\end{proof}
We recall that in the case $n=1$,   \eqref{bX1} appears in the paper
\cite{cal3}  of
Erich K\"ahler
 as equation (3) in \S 38. 
\enlargethispage{1cm}
\begin{corollary}\label{UNICUC} Under the $FC$-change of coordinates \eqref{zzZ},
  $x=\eta-V\bar{\eta},~ y=\xi-W\bar{\xi}$, the
  reproducing kernel \eqref{KHKX}  becomes $\mc{K}=K\circ FC$, 
\begin{equation}\label{MNKM}
\begin{split}
\mc{K}(\bar{\eta},\bar{V};\xi, W) & =(\det U)^{k/2}\exp{\mathcal{F}}, \quad
\emph{\text{where}} \\
~2\mathcal{F} & =\mathcal{F}_0+\Delta\mathcal{F};\\
~\mathcal{F}_0 & =\bar{\xi}^t\xi+\bar{\eta}^t\eta-\bar{\xi}^tW\bar{\xi}-\eta^t\bar{V}\eta,
\\
\Delta\mathcal{F}  & =(\bar{\zeta}^t-\zeta^t\bar{V})U(\xi-W\bar{\xi})+
(\bar{\eta}^t-\eta^t\bar{V})U(-\zeta+W\bar{\zeta}); \zeta=\eta-\xi .
\end{split}
\end{equation}
In particular, for $\xi=\eta, V=W$,  we have $\Delta\mathcal{F} =0$,  and 
\begin{equation}\label{LUI}
\mc{K}=\det (M)^{\frac{k}{2}}\exp (\mathcal{F}), \quad
\emph{\text{where}}  ~~\mathcal{F}=
\bar{\eta}^t\eta-\frac{1}{2}\eta^t\bar{W}\eta-\frac{1}{2}\bar{\eta}^tW\bar{\eta}, 
\end{equation}
 and the scalar product \eqref{ofiXX} becomes
\begin{equation}\label{XofiXX}
\begin{split}(\phi ,\psi ) & = \Lambda_n \int_{\eta\in\C^n;
\un-W\bar{W}>0}\bar{f}_{\phi}(\eta,W)f_{\psi}(\eta,W)\rho_2 \dd \eta \dd W,\\
\rho_2 & 
=\det(\un-W\bar{W})^ q\exp (-\mathcal{F}), q=k/2-n-1,
\end{split}
\end{equation}
with $\mc{F}$ given by \eqref{LUI}.

Also we have the relation
$$- \frac{\pa F}{\pa w_{ij}}=\frac{\pa \mathcal{F}}{\pa w_{ij}}= -
\bar{\eta}_i\bar{\eta}_j+\frac{1}{2}\bar{\eta}_i\bar{\eta}_j\delta_{ij},
\quad\emph{\text{or }~~~}\nabla F= -\nabla \mc{F} =
\frac{1}{2}\bar{\eta}\otimes\bar{\eta} . $$
\end{corollary}
\begin{proof} Formula (\ref{MNKM}) is obtained by brute force
calculation, using the relation $UW\bar{V}=U-\un$. 

 In the expression (\ref{kul}), we make the change of
variables (\ref{zzZ}), and we get easily the expression of $\mathcal{F}$ given in
(\ref{LUI}). 

In order to get the factor $\rho_2$ in the expression (\ref{XofiXX}),
we  firstly note that  
 $$(z^t,\bar{z}^t)=(\eta^t,\bar{\eta}^t)
\left(\begin{array}{cc} \un &-\bar{W}\\- W & \un\end{array}\right).$$
Now, we observe that if we have a linear transformation  $y=Cx$ of the
column $n$-vectors $y$ and $x$, then 
$$y_1\wedge\dots \wedge y_n=\det(C) x_1 \wedge\dots \wedge x_n.$$
We apply the formula $$\det\left(\begin{array}{cc} A & B\\ C & D \end{array}\right)
=\det A \det(D-CA^{-1}B) .$$
\end{proof}

\section{Classical motion  and quantum evolution on Siegel-Jacobi domains}\label{CLSQ1}

Let $M=G/H$ be a homogeneous  manifold with the $G$-invariant K\"ahler two-form
$\omega$ \eqref{kall}.
The energy function $\mc{H}$ (the classical Hamiltonian, or the covariant 
  symbol)  attached to
  the quantum Hamiltonian $\mb{H}$ is \cite{berezin}
    $$\mc{H}(z,\bar{z})=<e_{\bar{z}},e_{\bar{z}}>^{-1}<e_{\bar{z}}|\mb{H}|e_{\bar{z}}>.$$
 
Passing on from the dynamical system problem
 in the Hilbert space $\Hi$ to the corresponding one on $M$ is called
sometimes {\it dequantization}, and the dynamical system on $M$ is a classical
one \cite{sbcag,sbl}. Following Berezin \cite{berezin2,berezin1}, the
motion on the classical phase space can be described by the local
equations of motion
\begin{equation}\label{CLBER}
\dot{z}_{\alpha}=\i \left\{\mc{H},z_{\alpha}\right\},
  ~\alpha \in \Delta_+ ,
\end{equation}
where  the Poisson bracket is
  defined as 
$$\{f,g\}= \sum_{\alpha,\beta\in\Delta_+}(\mc{G}^{-1})_{
\alpha,\beta}(\frac{\pa f}{\pa z_{\alpha}}\cdot \frac{\pa g }{\pa
\bar{z}_{\beta} }- \frac{\pa f}{\pa \bar{z}_{\alpha}}\cdot \frac{\pa g
}{\pa z_{\beta} }),$$ 
 $f,g\in C^{\infty}(M)$ and $(\mc{G}^{-1})_{\alpha,\beta}$ are the
matrix elements of the inverse of the matrix $\mc{G}$ defined in
\eqref{Ter}.

The classical equations of motion \eqref{CLBER} on  the manifold $M$
can be written down as 
\begin{equation}
\i\left(\begin{array}{cc} \zn & \mc{G}\\ -\bar{\mc{G}} &
    \zn\end{array}\right)\left(\begin{array}{c}
    \dot{z}\\ \dot{\bar{z}}\end{array}\right) = -
  \left(\begin{array}{c}{\pa }/{\pa z} \\ {\pa }/{\pa \bar{z}}\end{array}\right)\mc{H}.
\end{equation}

We consider an algebraic Hamiltonian linear in the generators
${{X}}_{\lambda} $ of the
group of symmetry $G$
\begin{equation}\label{lllu}
H=\sum_{\lambda\in\Delta}\epsilon_{\lambda}{{X}}_{\lambda} ,
\end{equation}
and we also consider the associated operator $\mb{H}$. The energy
function associated to the Hamiltonian \eqref{lllu} is
$$\mc{H}(z,\bar{z})=\sum_{\lambda\in\Delta}\epsilon_{\lambda}\frac{(e_{\bar{z}},\mb{X}_{\lambda}e_{\bar{z}})}{(e_{\bar{z}},e_{\bar{z}})}.$$
Suppose that  $\mb{X}_{\lambda}\in\got{D}_1$, i.e.  the differential action corresponding to the operator
$\mb{X}_{\lambda}$ in (\ref{lllu}) can be expressed in a local
system of coordinates as a holomorphic  first order differential
operator with polynomial coefficients of the type \eqref{sss} 
\begin{equation}\label{VBC}\db{X}_{\lambda}=P_{\lambda}+\sum_{\beta\in\Delta_+}Q_{\lambda,
    \beta}\partial_{\beta}, \lambda\in\Delta.
\end{equation}
If  $P_{\lambda},~ Q_{\lambda,
    \beta}$ are the polynomials attached to the operator
  ${\mb{X}}_{\lambda} $ in \eqref{VBC}, let $\tilde{P}_{\lambda},~\tilde{Q}_{\lambda,
    \beta}$ be the polynomials attached to the operator
  ${\mb{X}}_{\lambda}^{\dagger}$ for the generator ${{X}}_{\lambda}\in\got{g}$   appearing in \eqref{lllu}.

With \eqref{Ter}, we get $$\frac{\pa \mc{H}}{\pa
  \bar{z}_{\beta}}=\sum_{\lambda,\gamma}\epsilon_{\lambda}Q_{\lambda,\gamma}\mc{G}_{\gamma,\beta}, $$
and the  classical motion generated
by the linear Hamiltonian  (\ref{lllu}) is  given by the  equations
of motion  on $M=G/H$ \cite{sbcag,sbl}: 
\begin{equation}\label{moveM}
{\i\dot{z}_{\alpha}=\sum_{\lambda\in\Delta}\epsilon_{\lambda}Q_{\lambda
,\alpha}},~\alpha\in\Delta_+ . 
\end{equation}
Similarly, the equations of motion on $M=G/H$ determined by a
Hamiltonian \begin{equation}\label{llluT}
\mb{H}=\sum_{\lambda\in\Delta}\epsilon_{\lambda}{\mb{X}}^{\dagger}_{\lambda} ,
\end{equation}
are
\begin{equation}\label{moveMT}
{\i\dot{\tilde{z}}_{\alpha}=\sum_{\lambda\in\Delta}\epsilon_{\lambda}\tilde{Q}_{\lambda
,\alpha}},~\alpha\in\Delta_+.   
\end{equation}

We look also for the solutions of the Schr\"odinger equation  attached
to the Hamiltonian  $\mb{H}$  (\ref{lllu}) 
\begin{equation}\label{SCH}
\mb{H}\psi = \i \dot{\psi},~~
\quad{\mbox{where~~~}}\psi= \e ^{\i \varphi}<e_z,e_z>^{-1/2}e_{\bar{z}}.
\end{equation}

Summarizing, in  the conventions at the beginning of \S \ref{DIFFAC}
and with the observation \eqref{bidiff}, we formulate the following  (see also \cite{sbl}, \cite{swA})  
\begin{Proposition}\label{MERETH}
On the homogenous  manifold $M = G/H $ on which  the generators
$X_{\lambda}\in\got{g}$ admit a holomorphic
representation as in  \eqref{VBC}, the classical
motion and the quantum evolution generated by the linear Hamiltonian
\emph{\text{(\ref{lllu})}}  are given by the same equation of motion
\emph{\text{(\ref{moveM})}}. The phase $\varphi$ in \emph{\text{(\ref{SCH})}} is given by the
sum $\varphi=\varphi_D +\varphi_ B$ of the
dynamical and Berry phase, 
\begin{subequations}
\begin{eqnarray}
\varphi_D & =& -\int_0^t\mc{H}(t)\dd t, \quad{\mbox{\emph{where}~~~}}\label{phi2}\\
\mc{H}(t) & = &\sum_{\lambda\in\Delta}\epsilon_{\lambda}\frac{(e_{\bar{z}},\mb{X}_{\lambda}e_{\bar{z}})}{(e_{\bar{z}},e_{\bar{z}})}\label{phiD1}\\
& = &\sum_{\lambda\in\Delta}\epsilon_{\lambda}(\tilde{P}_{\lambda}+
 \sum_{\beta\in\Delta_+}\tilde{Q}_{\lambda,\beta}\pa_{\beta}\ln <e_z,e_z>) \label{phiD2}\\
 & = &
\sum_{\lambda\in\Delta}\epsilon_{\lambda}\tilde{P}_{\lambda}+
\i\sum_{\beta\in\Delta_+}\dot{\tilde{z}}_{\beta}\pa_{\beta}\ln <e_{\bar{z}},e_{\bar{z}}>{\mbox{~}}\label{phiD3};\\
\varphi_B & = &   -\Im\int_0^t<e_{\bar{z}},e_{\bar{z}}>^{-1}<e_{\bar{z}}|\dd|e_{\bar{z}}>\label{phiB}\\
 & = & \nonumber
 \frac{\i}{2}\int_0^t\sum_{\alpha\in\Delta_+}(\dot{z}_{\alpha}\pa_{\alpha}-\dot{\bar{z}}_{\alpha}\bar{\pa}_{\alpha} ) \ln <e_{\bar{z}},e_{\bar{z}}>.
\end{eqnarray}
\end{subequations}
\end{Proposition}

\subsection{Equations of motion}\label{lasst}

Now we consider  a  Hamiltonian linear in the generators of
the group $G^J_n$ 
\begin{equation}\label{HACA}
\mb{H}= \epsilon_i\mb{a}_i+\overline{\epsilon}_i\mb{a}_i^{\dagger} +  
\epsilon^0_{ij}\mb{K}^0_{ij}+
\epsilon^-_{ij}\mb{K}^-_{ij}+\epsilon^+_{ij}\mb{K}^+_{ij}. 
\end{equation}
The hermiticity condition imposes to the matrices of  coefficients $\epsilon_{0,\pm}=(\epsilon^{0,\pm})_{i,j=1,\dots,n} $  the restrictions
\begin{equation}\label{CONDI}
\epsilon_0^{\dagger}=\epsilon_0; ~\epsilon_-=\epsilon_-^t;~
\epsilon_+=\epsilon_+^t; ~ \epsilon_+^{\dagger}=\epsilon_-.
\end{equation}
It is useful to introduce   the matrices $m,n,p,q\in\text{M}(n,\R)$
such that 
\begin{equation}\label{epsmn}
\epsilon_-=m+\i n, ~
  \epsilon_0^t/2=p+\i q; p^t=p; m^t= m; n^t =n; q^t=-q.
\end{equation} 

We shall describe the dynamics   on the Siegel-Jacobi ball (space)
$\mc{D}^J_n$  (respectively, $\mc{X}^J_n$) determined  by the linear
Hamiltonian \eqref{HACA}, \eqref{CONDI} and study the 
effect of the $FC$  ($FC_1$) transform on the equations of motion. We introduce some
notation. Let $W\in M(n,\C)$  be coordinates on a homogenous manifolds
$M=G/H$ 
- below $M$ will be one of the Siegel-Jacobi domains $\mc{D}_n$ or $\mc{X}_n$ -
and let 
$z\in\C^n$. We consider 
 a matrix Riccati equation on the manifold $M$ and a linear differential
equation in $z$
  \begin{subequations}\label{TOTAL}
\begin{align}
\dot{W} & =AW+WD+B+WCW, ~A,B,C,D\in M(n,\C); \label{RICC}\\
\dot{z} & = M+Nz; ~M= E+WF; ~ N= A+WC, ~E,F\in C^n. \label{LINZ}
\end{align}
\end{subequations}
\begin{Proposition} \label{POYT}
The classical motion and quantum evolution generated by  the linear
hermitian  Hamiltonian \eqref{HACA}, \eqref{CONDI}  are described by
first order differential equations:\\
a) on $\mc{D}^J_n$,  $(z,W)\in \C^n\times\mc{D}_n$ verifies
\eqref{TOTAL}, 
with coefficients
 \begin{subequations}\label{hip}
\begin{align}
A_c & =  -\frac{\i}{2}\epsilon_0^t, ~ B_c=-\i\epsilon_-, ~C_c=-\i\epsilon_+, ~
D_c= A_c^t ; \label{hip2}\\
E_c & =-\i\epsilon, ~F_c=-\i\bar{\epsilon}\label{hip1}.
\end{align}
\end{subequations}

b) on $\mc{X}^J_n$,  $(u,v)  \in\C^n\times\mc{X}_n$,  verifies
\eqref{TOTAL}, 
with coefficients  
\begin{subequations}\label{hipP}
\begin{align}
A_r & =
n+q,~B_r  = m-p, ~
C_r = -(m+p),~
D_r=
n-q;\label{hipP2} \\
E_r & = \Im \epsilon; F_r= -\Re\epsilon .\label{hipP1}
\end{align}
\end{subequations}

c) under the $FC$ transform \eqref{zzZ}, the differential equations 
 in the variables $\eta\in\C^n$, $W\in\mc{D}_n$ become independent:  $W$
verifies \eqref{RICC}  with coefficients \eqref{hip2} and $\eta $ verifies  
\begin{equation}\label{hipPRT1}
\i \dot{\eta} =  \epsilon +
\epsilon_-\bar{\eta} + \frac{1}{2}\epsilon_0^t\eta,~\eta\in\C^n.
\end{equation}

d) under the $FC_1$ transform, the equations in the variables
$\eta\in\C^n$, $v\in\mc{X}_n$  become  independent: $\eta$ verifies
\eqref{hipPRT1}, while
$v$ verifies \eqref{RICC} with coefficients \eqref{hipP2}.
\end{Proposition}
\begin{proof}
Firstly, we proof \eqref{hip}. With \eqref{nabla1} in Lemma \ref{PERKK},
 we get from (\ref{moveM}) the equations of motion for $z\in\C^n$:
$$\i
\dot{z}_{\alpha}=\epsilon_i\delta_{i\alpha}+\bar{\epsilon}_{i}w_{i\alpha}+\frac{\epsilon^+_{kl}}{2}(z_kw_{il}+z_lw_{ik})\delta_{i\alpha}+\frac{1}{2}\epsilon^0_{kl}z_k\delta_{l\alpha}.$$
The equations of motion  (\ref{moveM}) for $w_{pq}$, $W= (w_{pq})_{p,q=1,\dots,n}$, can be written down as
$$\i\dot{ w}_{pq}=\epsilon_{kl}Q_{kl,pq}. $$   With
\eqref{nabla2}, \eqref{nabla3} in  Lemma
\ref{PERKK} and \eqref{mircea}, we have for  $Q_{kl,pq}$-s the expressions:
\begin{eqnarray*}
\mb{K}^0_{kl} ~\rightarrow~Q_{kl,pq}^0~ & = &
w_{kp}\chi_{pl}\delta_{lq}+w_{kq}\chi_{ql}\delta_{lp}-w_{kl}\delta_{qp}\delta_{lp},\\
\mb{K}^+_{kl} ~\rightarrow~Q_{kl,pq}^+~ & = &  w_{qk}w_{pl}\chi_{pq}+w_{pk}w_{ql}\chi_{qp}-w_{pk}w_{pl}\delta_{pq} ,\\
\mb{K}^-_{kl} ~\rightarrow~Q_{kl,pq}^-~  & = &
\chi_{kl}(\delta_{kp}\delta_{lq}+\delta_{kq}\delta_{lp}-\delta_{kl}\delta_{pq}\delta_{kp}). 
\end{eqnarray*}
Explicitly,  the differential equations  for $(W,z)\in\mc{D}^J_n$ are
\begin{subequations}\label{Mhip}
\begin{eqnarray}
\i \dot{W}  & = &\epsilon_-  +
(W\epsilon_0)^s +W\epsilon_+W, ~ W\in\mc{D}_n,\label{Mhip2}\\
\i \dot{z} & = & \epsilon +
W\overline{\epsilon}+  \frac{1}{2}\epsilon_0^tz+W\epsilon_+z,
~z\in\C^n , \label{Mhip1}
\end{eqnarray}
\end{subequations}
and we have proved a).

Now we prove b). Firstly,  we prove \eqref{hipP2}. 

With the Cayley transform \eqref{big1}, we find for $\dot{W}$ the
value
$$2\i(v+\i\un)^{-1}\dot{v}(v+\i\un)^{-1}= A_cW+WD_c+B_c+WC_cW.$$
In the expression above we replace the value of $W$ as function of $v$
as given by \eqref{big1}, and we find a matrix Riccati equation on
$\mc{X}_n$ in $v$ of the form \eqref{RICC},
where the matrix coefficients $A_r$-$D_r\in M(n,\R)$ are expressed in function of
the coefficients $A_c$-$D_c$  as
\begin{equation}\label{AD1}
\begin{split}
A_r & =\frac{1}{2}(A_c-D_c+B_c-C_c) ,~
B_r  =\frac{1}{2\i}(A_c+D_c-B_c-C_c) ,\\
 C_r & = \frac{1}{2\i}(A_c+D_c+B_c+C_c) ,~
D_r  =\frac{1}{2}(-A_c+D_c+B_c-C_c) . 
\end{split}
\end{equation}
With \eqref{epsmn}, we  find the values given by \eqref{hipP2}.

Now we proof \eqref{hipP1}.  We differentiate \eqref{big2} and, with
\eqref{hip1}, we get successively
\[
\begin{split}
-2\dot{u} & =\dot{v}(\i z) + (v+\i \un)(\i \dot{z}) \\
 & = -2\dot{v}(v+\i\un)^{-1}u +  (v+\i\un)\times\\
 & \times\{\epsilon+  (v+\i\un)^{-1}
 (v-\i\un)\bar{\epsilon}
+[\frac{\epsilon_0^t}{2}+ (v+\i\un)^{-1} (v-\i\un)\epsilon_+]2\i  (v+\i\un)^{-1}\} u\\
 & = (v+\i\un)\epsilon+ (v-\i\un)\bar{\epsilon} +T  (v+\i\un)^{-1}u,
\end{split}
\]
where
$$T  = -2\dot{v}+\i  (v+\i\un)\epsilon_0^t+2\i  (v-\i\un)\epsilon_+.$$
With \eqref{hipP2} we get for $T$ the value
$$ T= [v(\epsilon_0^s+\!\epsilon_++\epsilon_-)+
\i(-\epsilon_0^a+\epsilon_--\epsilon_+)](v+\i\un), $$
and we obtain  \eqref{hipP1}. Explicitly, b) can be write down as 
\begin{subequations}
\begin{align}
-2 \dot{v}  & = \epsilon_0^s -\!(\epsilon_-+\epsilon_+)
+\i v(\epsilon_0^a+\epsilon_--\epsilon_+)+
\label{MhipP2}\\
\nonumber &  \i(-\epsilon_0^a+\epsilon_--\epsilon_+)v+
v(\epsilon_0^s+\epsilon_-+\epsilon_+)v;\\
-2\dot{u} &\! =\!  v(\epsilon\!+\!\bar{\epsilon})+\i(\epsilon\!-\!\bar{\epsilon})+
[v(\epsilon_0^s+\!\epsilon_++\epsilon_-)+
\i(-\epsilon_0^s+\epsilon_--\epsilon_+)]u.\label{MhipP1} 
\end{align}
\end{subequations}

Below we prove  \eqref{hipPRT1}. We take the derivative of $\eta$ in
the $FC^{-1}$ transform \eqref{invFC},  and we get
$$\dot{\eta} = (\un-W\bar{W})^{-1}(\dot{W}\bar{W}+W\dot{\bar{W}})\eta
+
 (\un-W\bar{W})^{-1}(\dot{z}+\dot{W}\bar{z}+W\dot{\bar{z}}).$$
Then we introduce for $\dot{W}$ and $\dot{z}$ the values from
\eqref{Mhip2} and  respectively \eqref{Mhip1}  and we pass from $z$ to $\eta$ with
\eqref{zzZ}:
\[
\begin{split}
\i (\un-W\bar{W})\dot{\eta} & = \{
[\epsilon_-+(W\epsilon_0)^s+W\epsilon_+W]\bar{W}
-W[\bar{\epsilon}_-+(\bar{W}\bar{\epsilon}_0)^s+\bar{W}\bar{\epsilon}_+\bar{W}]
\}\eta + \\
& [\epsilon_-+(W\epsilon_0)^s+W\epsilon_+W](\bar{\eta}-\bar{W}\eta)
-W[\bar{\epsilon}+\bar{W}{\epsilon}+(\frac{\bar{\epsilon}_0^t}{2}+\bar{W}\bar{\epsilon}_+)(\bar{\eta}-\bar{W}{\eta})] +\\
&
\epsilon+W\bar{\epsilon}+(\frac{\epsilon_0^t}{2}+W\epsilon_+)(\eta-W\bar{\eta}) .\end{split}
\]
We calculate the coefficient of $\eta$ as
$\frac{1}{2}(\un-W\bar{W})\epsilon_0^t$, those of $\bar{\eta}$ is
$(\un-W\bar{W})\epsilon_-$, and we get  \eqref{hipPRT1}. 

 In order to
avoid the longer calculation of introducing \eqref{hipP1}  in
\eqref{XCX1},  the motion in $(\eta, v)\in(\C^n,\mc{X}_n)$ at d)  is
obtained putting together \eqref{hipPRT1} and  \eqref{hipP2}.
\end{proof}
Note that   starting  with the {\it Hamiltonian} \eqref{HACA},  {\it linear in the
generators of the Jacobi group } $G^J_n$, {\it the equation of motion
\eqref{hip2} (\eqref{hipP2}) on the
Siegel ball} $\mc{D}_n$ ({\it Siegel upper half plane}
$\mc{X}_n$)  {\it depends {\bf only} on the generators of de group}
$\text{Sp}(n,\R)_{\C}$(respectively, $\text{Sp}(n,\R)$).

\subsection{Solution of the equations of motion}\label{solution}
We  solve the system of differential equations on the Siegel-Jacobi
domains   appearing in Proposition \ref{POYT}. 

{\bf a.} Firstly, we recall  how {\bf to solve the matrix Riccati
equation} \eqref{RICC}  {\bf by linea\-rization}.
{\it If we proceed to the homogenous coordinates} $W=XY^{-1}$, $X,Y\in
M(n,\C)$,  {\it a linear
system of ordinary differential equations is attached to the matrix
Riccati equation} \eqref{RICC} (cf. \cite{levin}, see also \cite{sbl})
\begin{equation}\label{RICClin}
\left( \begin{array}{c}\dot{X}\\\dot{Y}\end{array}\right)=h
\left( \begin{array}{c}{X}\\{Y}\end{array}\right), ~
h=\left(\begin{array}{cc} A & B\\ -C & -D \end{array}\right) . 
\end{equation}
{\it Every solution of} (\ref{RICClin}) {\it is a solution of} (\ref{RICC}),
{\it whenever} $\det (Y)\not=0$. 

 For the motion  on $\mc{D}_n$, the matrix elements  of $h$ in \eqref{RICClin},
denoted  $h_c$, 
are  given in \eqref{hip2},  and 
the linear system of differential equations  attached to the matrix
Riccati equation \eqref{Mhip2} is
 \begin{equation}\label{Rlin}
\left( \begin{array}{c}\dot{X}\\\dot{Y}\end{array}\!\right)=h_c
\left(\! \begin{array}{c}{X}\\{Y}\end{array}\right), ~
h_c=\left(\begin{array}{cc} -\i(\frac{\epsilon_0}{2})^t &
    -\i\epsilon_-\\ \i\epsilon_+& 
\i\frac{\epsilon_0}{2}\end{array}\right), ~W=X/Y\in\mc{D}_n. 
\end{equation}
Due to the conditions   (\ref{CONDI}), {\it the matrix $h_c$ has the form} 
 \eqref{xXC},   {\it and
we conclude that} $h_c \in \got{sp}(n,\R)_{\C}$.  The action of  the 
element $g\in\text{Sp}(n,\R)_{\C}$ \eqref{dgM} on $W\in\mc{D}_n$ is
given by the linear fractional transformation (\ref{x44}). 

 Now  we look at the
matrix Riccati equation \eqref{RICC} on  the Siegel upper half plane $\mc{X}_n$ in the symmetric
variables $v$  with coefficients $A_r-D_r$ given by \eqref{AD1} -- or equation \eqref{MhipP2}.
We associate to the matrix
Riccati equation \eqref{MhipP2} a
linear system of first order differential equations of the type
(\ref{RICClin})  with a matrix coefficients $h_r$
\begin{equation}\label{hrr}
\!\left( \!\begin{array}{c}\!\dot{X}\\\dot{Y}\!\end{array}\!\right)\!=\!h_r
\!\left( \!\begin{array}{c}\!{X}\\{Y}\!\end{array}\right)\!, h_r\!=\! \left(\begin{array}{cc}\! A_r &
    B_r\\ -C_r & -D_r\! \end{array}\right)\!= \left(\begin{array}{cc} \!n+q
    & m-p\\  m+ p  & -n+q \!\end{array}\!\right)\!, \!v=X/Y\in\mc{X}_n.
\end{equation}
  Due to the hermiticity conditions (\ref{CONDI}), {\it the
matrix} $h_r$ {\it verifies the conditions} (\ref{XREAL}), because of
\eqref{AD1},  and  {\it we  see  that} 
$h_r\in\got{sp}(n,\R).$  The matrices $h_r$ and $h_c$ are
related by relations of the type \eqref{MNB}, \eqref{MNB1},
i.e. $h_c=(h_r)_{\C}$
and have the same eigenvalues.

 Using general considerations of solving systems of  first order linear differential
equations \cite{hart} applied to $\got{sp}(n,\R)$ \cite{mey} and  the considerations above,
we formulate 
\begin{Remark}\label{scroafa}
The linear system
of first order differential equations \eqref{Rlin} (\eqref{hrr})  associated to the matrix
Riccati equations \eqref{Mhip2}  (\eqref{MhipP2})
on $\mc{D}_n$ ($\mc{X}_n$) describes   the
  time-dependent vector field induced by the infinitesimal action of
  the group $\text{\emph{Sp}}(n,\R)_{\C}$ (respectively, $\text{\emph{Sp}}(n,\R)$).
 \eqref{Rlin} (\eqref{hrr}) is a linear  Hamiltonian system in the
 meaning of   Meyer
 \cite{mey} or a canonical system in the sense of  Yakulovich \cite{ya}
  (Hamiltonian system) 
  (respectively,  in the sense Yakulovich \cite{ya}) and
  $h_c=(h_r)_{\C}$. 

The infinitesimal group action of $\got{sp}(n,\R)_{\C}$ (respectively, $\got{sp}(n,\R)$) is given by the Lie algebras
homomorphism 
\begin{equation}\label{nuhom}
\nu_c : \got{sp}(n,\R)_{\C}\rightarrow \Ham(\mc{D}_n),\quad \nu_r : \got{sp}(n,\R)\rightarrow \Ham(\mc{X}_n).
\end{equation}
The infinitesimal group action associated to  \eqref{RICClin} is given by 
\begin{equation}\label{infa}
\nu \left(\begin{array}{cc} A & B\\ -C & -D \end{array}\right) =
-(B+AZ+ZD+ZCZ)_{im}\frac{\pa}{\pa w_{im}} . 
\end{equation}
  Let
\begin{equation}\label{UTT}
U(t,t_0)=\left(\begin{array}{cc} U_1(t,t_0) & U_2(t,t_0) \\
    U_3(t,t_0) & U_4(t,t_0)\end{array}\right)
\end{equation}
be the fundamental matrix  of the ordinary differential equation
\eqref{RICClin}, i.e. $\dot{U}=hU$ such that $U(t_0,t_0)=1$, where $h
= h_c$ ($h=h_r$) for the motion on $\mc{D}_n$ (respectively, $\mc{X}_n$). Then the fundamental
solution $U_c(t,t_0)$ ($U_r(t,t_0)$) corresponding to $h=h_c$  ($h=h_r)$ is a  ${\emph{\text{Sp}}}(n,\R)_{\C}$
(respectively ${\emph{\text{Sp}}}(n,\R)$)-matrix and 
$$W(t,t_0)=[U_1(t,t_0)W(t_0)+ U_2(t,t_0)][U_3(t,t_0) W(t_0)+
U_4(t,t_0)]^{-1}$$
is the solution of equation \eqref{RICC} with the initial condition
$W(t_0,t_0)=W(t_0)$. 
\end{Remark}

For selfcontainedness, we recall  the terminology used in Remark
\ref{scroafa}. The systems of linear  ordinary 
differential  equations
$\dot{z}=Az$, $A\in\got{sp}(n,\R)$ appear in the context of {\it linear
Hamiltonian systems} \cite{mey}. The eigenvalues of the Hamiltonian
matrices are described in  Remark \ref{rem77} e).
The Hamiltonian equations can be written as  the system of ordinary
differential equations 
\begin{equation}\label{hamil}
\dot{z}=J\nabla H,
\end{equation}
where $z^t=(q^t,p^t)$, $q,p\in\R^n$,  $H=H(t,q,p)$ is the
Hamiltonian, and here $\nabla H$$=$$(\frac{\pa H}{\pa z_1}$,$\dots$,
$ \frac{\pa H}{\pa z_{2n}} )$. The Hamiltonian system \eqref{hamil} is
called a {\it
  Hamiltonian linear system}  (cf \cite{mey},  also called {\it
  canonical}, cf \cite{ya}  p. 110), if  the Hamiltonian $H$ has the form
\begin{equation}\label{linie}
H=\frac{1}{2}z^tSz,~S\in M(2n,\R) ,~ S=S^t, 
\end{equation}
 and the system \eqref{hamil}  is written as the system of linear
ordinary differential equations 
\begin{equation}\label{alth}
\dot{z}=JS z= A z,~ A\in\got{sp}(n,\R), ~ S=S^t. 
\end{equation}
The equation $\dot{z}=Az$,  $A\in\got{sp}(n,\R)_{\C}$,  where  $A$ is expressed as in
\eqref{ciudat}, is also called {\it Hamiltonian}, cf.  \cite{ya} p. 111.

{\bf b.} Now we discuss the linear system of the type \eqref{alth}
associated to the {\bf  matrix Riccati equation} 
  \eqref{MhipP2}  {\bf in the case of $T$-periodic coefficients}. Let $\Delta(t)$
be the fundamental solution of  the equation \eqref{alth} with
periodic coefficients  such $\Delta(0)=\dn$. Then $\Delta(t+T)=\Delta(t)\Delta(T),
\forall t\in\R$.  We take over  {\it Floquet-Lyapunov Theorem}  and {\it Krein Gel'fand Theorem},
 (cf. Proposition 4.2.1, 
Theorem 3.4.2 and Corollary 3.4.1 in \cite{mey}; see also Ch II in
\cite{ya}), which we formulate as:
\begin{Remark}\label{rem12}
The {\it monodromy}  matrix $\Delta(T)$ is a
nonsingular, 
symplectic matrix.  The
  fundamental matrix solution of Hamiltonian equation \eqref{alth}  that
  satisfies $\Delta(0)=\dn$ is of the form $\Delta(t)=X(t)\exp(Kt)$, 
  where $X(t)$ is symplectic and $T$-periodic and $K$ is Hamiltonian. 
Real $X(t$) and $K$ can be found by taking $X(t)$ to be $2T$-periodic
if necessary. The change of variables $z= X(t)w$ transforms
the periodic Hamiltonian system \eqref{alth} to the constant
Hamiltonian system
\begin{equation}\label{kw}
\dot{w}=Kw,~ K\in\got{sp}(n,\R).
\end{equation}
The linear autonomous Hamiltonian system \eqref{kw}  is stable
iff:  (i) $K$ it  has only pure imaginary eigenvalues and (ii) $K$ is
diagonalizable (over the complex numbers).\\
The
system \eqref{kw}  is parametrically stable if and only if:
(i) All the eigenvalues of $K$ are pure imaginary $\pm\i\alpha_i$;
(ii) $K$ is nonsingular;
(iii)  The matrix $K$  is diagonalizable over the complex numbers;
(iv)  The restriction of the Hamiltonian $H$ to the linear space
generated by $\pm\i\alpha_j$ is positive or negative definite for each $j$.
\end{Remark}

{\bf c.}  Let us discuss {\bf  the solution of the matrix Riccati equation}
(\ref{RICC}) {\bf in the case of constant coefficients}. Let
$\lambda_1,\dots,\lambda_{2n}$ be the eigenvalues of characteristic
equation associated to the matrix $h$ defined in \eqref{RICClin}
\begin{equation}
\det (h-\lambda \dn)=0.
\end{equation}
Let 
\begin{equation}\label{486}
\Lambda =\left(\begin{array}{cc} \Lambda_1 & 0\\ 0 &\Lambda_2\end{array}\right),
\end{equation}
where $\Lambda_1$ ($\Lambda_2$) is a diagonal matrix with entries
$\lambda_i$, $i=1,\dots,n$ (respectively, $i=n+1,\dots, 2n$).

Let us suppose that the matrix $h$ has {\it a simple structure}  (cf.
Ch. III in \cite{gant} or $h$ has {\it simple elementary divisors} cf. \cite{ya}) and let $V$
be {\it the fundamental matrix}  \cite{gant} of $h$ \eqref{RICClin},
i.e.
\begin{equation}\label{DIAGV}
V h = \Lambda V. 
\end{equation} 
Let us consider a partition of $V$ into block form
\begin{equation}
V=\left(\begin{array}{cc} V_1 & V_2\\ V_3 & V_4,\end{array}\right),
\end{equation}
where $V_i, i=1,\dots,4\in M(n,\C)$.

Let $Q=Q(W,A,B,C,D)$ the matrix which appears
in the rhs of \eqref{RICC}, and let us denote by  $Q_0$ the value of
$Q$ at  $W_0=W(t=0)$.
\begin{Remark}\label{Rem5}
The autonomous linear system \eqref{RICClin} associated to the matrix
Riccati equation \eqref{RICC} has the standard solution
\begin{equation}\label{QWE} 
\left(\begin{array}{c} X\\Y\end{array}\right) = \e^{t h}
\left(\begin{array}{c} X_0\\Y_0\end{array}\right) . 
\end{equation}
The solution $W$ of the matrix Riccati
equation \eqref{RICC} has the  Taylor expansion 
$$W(t)=W_0+tQ_0+\frac{t^2}{2}[(A+W_0C)Q_0+Q_0(D+CW_0)]+\dots. $$

If the matrix $h$ \eqref{RICClin}  has a simple structure, then  the solution of the matrix Riccati equation \eqref{RICC} with constant
coefficients with the initial condition $W(0)=W_0$  is
\begin{subequations}\label{VRRR}
\begin{eqnarray}
~ W & = & (V_1-W'V_3)^{-1}(W'V_4-V_2), \label{VR1}\\
~ W' & = & W'(t,W'_0) =
\e^{t\Lambda_1 }W'_0\e^{-t\Lambda_2 },\label{VR2}\\
 ~ W'_0 & = & (V_1W_0+V_2)(V_3W_0+V_4)^{-1}.\label{VR3}
\end{eqnarray}
\end{subequations}
Closed paths on the Siegel ball $\mc{D}_n$ ($W(T)=W(0)$) can be obtained if the
imaginary eigenvalues  $\lambda_i$ of the matrix $h$ are rationally
commensurable and different and $T$ is the least common multiple of
$2\pi|\lambda_m-\lambda_i|$, $i=1,\dots,n$,  $m=n+1,\dots,2n$. 

Suppose now we are in the case when the autonomous Hamiltonian
system \eqref{Rlin} is stable: the matrix $h_r$ has $2n$
purely imaginary distinct  eigenvalues and  the matrix $h_r$ can be
put in the form
\begin{equation}\label{hr}
h_r = \left(\begin{array}{cc}\zn &  d\\ -d & \zn \end{array}\right),
\quad d=\emph{\text{diag}}(\alpha_1,\dots,\alpha_n).
\end{equation}
Correspondingly, in \eqref{486} we get 
\begin{equation}\label{lamb12}
\Lambda_1=\i d, ~~\Lambda_2=-\i d
\end{equation}
which have  to be introduced in \eqref{VRRR}  for the motion on
$\mc{X}_n$ in the case of the constant coefficients in \eqref{hipP2}.

We also  are in the case where the matrix $h_c$ has 2n
distinct pure imaginary eigenvalues and  $h_c$ can be put into  the form 
\begin{equation}\label{hc}
h_c = \left(\begin{array}{cc}\i d  &  \zn \\ \zn  & -\i d \end{array}\right),
\quad d=\emph{\text{diag}}(\alpha_1,\dots,\alpha_n),
\end{equation}
 with the same
eigenvalues as in \eqref{lamb12}. 
\end{Remark}
\begin{proof} As was mentioned in  Remark \ref{scroafa}  a),  the linear system of first order differential equations
(\ref{RICClin}) is associated with the matrix Riccati equation
(\ref{RICC}). A solution to (\ref{RICClin}) projects to a solution to
(\ref{RICC}) via the map $\Psi(X,Y)=XY^{-1}$.
The autonomous linear system (\ref{RICClin})  has the standard solution
\eqref{QWE} \cite{hart}.
We recall that if $A\in M(n,\mathbb{F})$, ($\mathbb{F}=\R$ or $\C$) has the
minimal polynomial $\psi(\lambda)=(\lambda-\lambda_1)^{m_1}
(\lambda-\lambda_2) ^{m_2}\cdots (\lambda-\lambda_s)^ {m_s}$, where
$\lambda_1,\dots,\lambda_s$ are characteristic roots of $A$,  then 
$$\e^{At}=\sum_{k=1}^s[Z_{k1}+Z_{k2}t+\dots +Z_{km_k}t^{m_k-1}]\e^{\lambda_kt}.$$
The matrices $Z_{kj}$ are completely determined by $A$ \cite{gant}. In
particular, if the minimal polynomial has only simple roots,  the
Lagrange-Sylvester interpolation formula gives for $A\in M(n,\C)$ \cite{gant}
$$\e^{At}=\sum_{k=1}^n\frac{(A-\lambda_1\un)\cdots(A-\lambda_{k-1}\un)\cdots(A-\lambda_{k+1}\un)\cdots(A-\lambda_{n}\un)}
{(\lambda_k-\lambda_1)\cdots (\lambda_k-\lambda_{k-1})\cdots (\lambda_k-\lambda_{k+1})\cdots (\lambda_k-\lambda_n)}\e^{\lambda_kt}.$$

In the situation of Remark \ref{rem12},  we diagonalize the matrix $h$ via  (\ref{DIAGV}) and we make a
change of variables
\begin{equation}
\left(\begin{array}{c} X' \\Y'\end{array}\right) = 
\left(\begin{array}{cc} V_1 & V_2\\ V_3 & V_4\end{array}\right)
\left(\begin{array}{c} X\\Y\end{array}\right) . 
\end{equation}
The system (\ref{RICClin}) in the new variables  reads
$$\dot{X}'=\Lambda_1 X'; ~ \dot{Y}'=\Lambda_1 Y', $$
with the solution $$X'=\e^{t \Lambda_1}X'_0; Y'= \e^{t
  \Lambda_2}Y'_0.$$
We calculate  $W'=X'{Y'}^{-1}$ and get the expression (\ref{VR2}),
then we calculate $W'_0=X'_0(Y'_0)^{-1}$ and we obtain (\ref{VR3}), and
finally, the equation (\ref{VR1}). In the situation of Remark
\ref{rem12}, $\Lambda_1=\i d$, $\Lambda_2=-\i d$.

The assumption contained in \eqref{hr} with the consequence
\eqref{lamb12} is expressed in Proposition 3.1.18 in \cite{am} and
\eqref{hc} appears in
Remark \ref{rem77},  e).  Then we
apply the transform \eqref{pana} and the assertion for
$\got{sp}(n,\R)_{\C}$ follows.
\end{proof}

{\bf d.} Finally, we discuss {\bf  how to
solve the decoupled system  at c) in Proposition \ref{POYT}}. In (\ref{hipPRT1}) we
introduce $\eta=\xi - \i \zeta, \xi,\zeta\in\R^n$, we put 
$\epsilon= b+\i a$, where $a,b\in\R^n$.  The first order complex differential equation equation
(\ref{hipPRT1}) is equivalent with a system of first order real 
differential equations with real coefficients, which we write as  
\begin{equation}\label{LINe} \dot{Z}=h_rZ + F, ~ Z = 
\left( \begin{array}{c}{\xi} \\
    {\zeta}\end{array}\right), ~ 
 F = \left(\begin{array}{c}a \\ b \end{array}\right).
\end{equation}
 Because of \eqref{CONDI}, the matrices $p,m,n$ are symmetric, while $q$
is antisymmetric, as in \eqref{epsmn}.  This means that the matrix $h_r$ has the form
\eqref{XREAL}, i.e. $h_r\in\got{sp}(n,\R)$, as was already underlined
in Remark \ref{scroafa}. With general considerations on solving first order  linear systems of 
nonhomogenous equations \cite{hart}  particularized to the case of
Hamiltonian  matrices \cite{mey}, we get 
\begin{Remark}
Let $\Delta(t,t_0)$ be the fundamental solution of the homogeneous
equation $\dot{Z}= h_r Z$ associated with \eqref{LINe}, where $h_r\in\got{sp}(n,\R) $.  The
fundamental matrix solution $\Delta(t,t_0)$ of the linear Hamiltonian
system is symplectic. The solution of the inhomogeneous
equation \eqref{LINe} is
$$Z(t)=\Delta (t,t_0)Z(t_0)+\int_{t_0}^t\Delta
(t,\tau)F(\tau)\dd\tau. $$
In the case of constant coefficients, the solution of \eqref{LINe}
is
$$Z=\e^{h_r(t-t_0)}Z_0+ \int_{t_0}^t\e^{h_r(t-\tau)}F(\tau)\dd\tau.$$
\end{Remark}
So, solving the equation (\ref{LINe}), we find  $\eta=\xi-\i\zeta$,
and we  find the solution of (\ref{hipPRT1}).

The solution of
(\ref{LINZ})  with coefficients \eqref{hipP}, i.e. \eqref{Mhip1},  is
obtained via the $FC$ transform  $z=\eta-W\bar{\eta}$. 

\section{Phases}\label{FAZE}
\subsection{Berry phase for $\mathcal{D}^J_n$}
 Formula \eqref{phiB}  for the Berry  phase  on the
Siegel-Jacobi ball 
$\mc{D}^J_n$ in the variables $(W,z)\in\mc{D}_n\times\C^n$ reads:
$$\frac{2}{\i}\dd\varphi_B = (\sum \dd w_{ij}\frac{\pa }{\pa w_{ij}}- cc )f
+(\sum \dd z_i\frac{\pa }{\pa z_i}-cc )f, $$
where  $(X-cc)$ means $(X-\bar{X})$. 
Above $f$ is the K\"ahler potential (\ref{kelerX})
written as $$f=-\frac{k}{2}\log (\det(\un-W\bar{W})) + F.$$
With \eqref{DERCV}, \eqref{ultiM} and \eqref{derA}, 
we have
\begin{subequations}\label{again}
\begin{eqnarray}
\frac{\pa f}{\pa z_i} & = & \bar{\eta}_i,\label{again1}\\
\frac{\pa f}{\pa w_{ik}} & = &
\frac{k}{2}(2X_{ik}-X_{ik}\delta_{ik})+\bar{\eta}_i\bar{\eta}_k-
\frac{1}{2}\bar{\eta}_i\bar{\eta}_k\delta_{ik}, \quad{{\text{or
  }~~~}} \nabla f=\frac{1}{2}(kX+\bar{\eta}\otimes\bar{\eta} ) . \label{again2}
\end{eqnarray}
\end{subequations}

With the $FC$-transform \eqref{zzZ} $z_i=\eta_i-w_{ij}\bar{\eta}_j$,  we find 
\begin{Remark}
The Berry phase on the Siegel-Jacobi ball $\mc{D}^J_n$ is expressed
in the variables $(W,\eta)\in\mc{D}_n\times\C^n$ as 
\[
\begin{split}
\frac{2}{\i}\dd\varphi_B  & =
\{[\sum\frac{k}{2}(2X_{ij}-X_{ij}\delta_{ij})-\frac{1}{2}\bar{\eta}_i\bar{\eta}_j\delta_{ij}]\dd
w_{ij}-cc\}
  + [(\bar{\eta}_i+\bar{w}_{ij}\eta_j)\dd \eta_i-cc]\\
 & = 
\{\frac{k}{2}[2\tr(X\dd
W)-\tr(\dia(X)\dia(\dd
W))]-\frac{1}{2}\bar{\eta}^t\dia(\dd W)\bar{\eta}-cc\}\\
& +[\dd\bar{\eta}^t(\bar{\eta}+\bar{W}\eta)-cc].
\end{split}
\]
\end{Remark}

  \subsection{Dynamical phase}   We calculate the energy function
  attached to the Hamiltonian (\ref{HACA}) using the formula
  (\ref{phiD1}). We write down   $\mc{H} =
  \mc{H}_1+ \mc{H}_2$, where  $\mc{H}_1$ is the first term $\mc{H}_1=
\sum_{\lambda\in\Delta}\epsilon_{\lambda}\tilde{P}_{\lambda}$  in
\eqref{phiD2}, while $ \mc{H}_2$ is the rest in \eqref{phiD2}, which
corresponds to 
$\i\sum_{\beta\in\Delta_+}\dot{\tilde{z}}_{\beta}\pa_{\beta}f$, but we
prefer the brute force calculation with the formulae \eqref{Ppark}. With
\eqref{Ppark}  and \eqref{again1}, we get 
\begin{equation}\label{gfgh}
\mc{H}_1    =   {\epsilon}^tz+\frac{k}{4}\tr
(\epsilon_0)+\frac{k}{2}\tr(\epsilon_-W)+\frac{1}{2}z^t\epsilon_-z . 
\end{equation}
For $\mc{H}_2$, we get with \eqref{Ppark}
\begin{subequations}\label{H2X}
\begin{align}
\mc{H}_2  & =
\epsilon^tW\bar{\eta}+\bar{\epsilon}^t\bar{\eta}+\frac{1}{2}z^t\epsilon^t_0\bar{\eta}+z^t\epsilon_-W\bar{\eta}+\mc{H}_3,\\
\mc{H}_3  & = (\epsilon^0_{kl}w_{li}\nabla_{ik}+
\epsilon^+_{kl}\nabla_{kl}+\epsilon^-_{kl}w_{\alpha
l}w_{ki}\nabla_{i\alpha})f, 
\end{align}
\end{subequations}
where $\nabla f$ has  the value
 \eqref{again2}. 

Let us denote by $\Lambda =\Lambda(W,\epsilon_-,\epsilon_+,\epsilon_0)$
  the matrix appearing in  the rhs of \eqref{Mhip2}, i.e. $\Lambda =
  \i Q$. We get
\begin{equation}\label{chi3}
\mc{H}_3 = \frac{1}{2}[k\tr(\tilde{\Lambda} X)+\bar{\eta}^t\tilde{\Lambda}\bar{\eta}],
\quad{\text{where}}~ \tilde{\Lambda} = \Lambda(W,\epsilon_+,\epsilon_-,\epsilon^t_0)
=\epsilon_++(\epsilon_0W)^s+W\epsilon_-W. 
\end{equation} 
Now we apply  the $FC$-transform \eqref{zzZ} in \eqref{gfgh}   and
\eqref{H2X} and we express  the energy
function as sum of two separated terms in the  independent  variables
$\eta\in\C^n$ and $W\in\mc{D}_n$
\begin{subequations}\label{enf}
\begin{align}
\mc{H} & = \mc{H}_{\eta}+ \mc{H}_{w},\\
\mc{H}_{\eta} & =\epsilon^t\eta +
\bar{\epsilon}^t\bar{\eta}+\frac{1}{2}(\eta ^t\epsilon_- \eta
+\bar{\eta}^t\epsilon_+\bar{\eta} + 
\bar{\eta}^t\epsilon_0 \eta   ) , \\
\mc{H}_{w} & = \frac{k}{2}\tr\{(\epsilon_0)^s+ [W\epsilon_-+\epsilon_+\bar{W}+(\epsilon_0W)^s\bar{W}](\un-W\bar{W})^{-1}\} .
\end{align}
\end{subequations}
We calculate the critical points of the energy function
\eqref{enf} attached to linear hermitian Hamiltonian \eqref{HACA}.
We get 
\begin{subequations}\label{grad}
\begin{align}
\nabla \mc{H}_{w}  & =
2(\un-\bar{W}W)^{-1}\bar{\tilde{\Lambda}}(\un-W\bar{W})^{-1},\\
\frac{\pa \mc{H}_{\eta}}{\pa \eta} & = \epsilon +\epsilon_-\eta +\frac{1}{2}\epsilon^t_0\bar{\eta}.
\end{align}
\end{subequations}
We have proved 
\begin{Remark} The energy function  in  the variables $(\eta,W)\in\C^n\times\mc{D}_n$ attached to the linear hermitian
  Hamiltonian  \eqref{HACA} is given by \eqref{enf}. The critical
  points $W_c$ of the energy functions \eqref{enf} are the
solutions of the algebraic matrix Riccati equations $\bar{\tilde{\Lambda}}
=0$, i. e. the solutions of the equation  $\dot{W}= 0$ in
\eqref{Mhip2}.  The critical value $\eta_c$
corresponds to the solution $\dot{\xi}=0$ , $-\dot{\zeta} =0$ in 
\eqref{LINe}.   
\end{Remark}

{\bf Acknowledgement.}   
 The question of the fundamental conjecture for the
Siegel-Jacobi disk was raised to me by Pierre Bieliavsky.
I am grateful to Mircea Bundaru for useful discussions and 
 for bringing into my attention formula \eqref{mircea}.  
This investigation  was  supported by the 
ANCS project  program PN 09 37 01 02/2009 and  by the UEFISCDI-Romania
 program PN-II Contract No. 55/05.10.2011.

\today
\end{document}